\documentclass[paper=a4,headings=normal,bibliography=toc,index=totoc,
	]{scrartcl}
	
\usepackage[utf8]{inputenc}
\usepackage{lmodern}
\usepackage[T1]{fontenc}
\usepackage[german,english]{babel}

\ifpdfoutput{\usepackage{microtype}}{}

\usepackage{fixltx2e}  

\usepackage{amsmath,amsfonts,amssymb,amsthm}

\usepackage{textcomp}
\usepackage{exscale} 

\usepackage{cancel}   
\usepackage{graphicx} 
\usepackage{url} 
\usepackage[arrow, matrix, curve]{xy}
\usepackage{xspace}
 
\usepackage{soul}

\usepackage[makeindex]{splitidx}
\newindex[Index of notation]{not}
\newindex[Index]{idx}

\usepackage[usenames,dvipsnames]{xcolor}

\ifpdfoutput{
\usepackage[pdftex,hypertexnames=false,plainpages=false,pdfpagelabels=true,setpagesize=false]{hyperref}\hypersetup{breaklinks=true,linktocpage=false,colorlinks=true,urlcolor=black,linkcolor=black,citecolor=black,pdfauthor={Julian Hofrichter}
}}{}

\newcommand{\bel}[1]{\begin{equation}\label{#1}}

\newcommand{\be}{\begin{equation}}
\newcommand{\ba}{\begin{eqnarray}}
\newcommand{\ea}{\end{eqnarray}}
\newcommand{\rf}[1]{(\ref{#1})}
\newcommand{\bi}{\bibitem}
\newcommand{\qe}{\end{equation}}

\newcommand{\R}{\mathbb{R}}

\newcommand{\N}{\mathbb{N}}
\newcommand{\Np}{\mathbb{N}_+}
\newcommand{\Z}{\mathbb{Z}}

\newcommand{\LII}{\mathcal{L}^2}
\newcommand{\D}{\Delta}

\newcommand{\map}{\longrightarrow}
\newcommand{\too}{\longrightarrow}

\newcommand{\half}{\frac12}

\newcommand{\sumstack}[1]{\sum_{\substack{#1}}}
\newcommand{\smin}{\setminus}

\providecommand{\ce}{\mathrel{\mathop:}=}
\providecommand{\ec}{=\mathrel{\mathop:}}
\providecommand{\ceq}{\mathrel{\mathop:}\Leftrightarrow}
\providecommand{\co}{\colon}
\providecommand{\itind}[1]{\textit{#1}\index{#1}}
\providecommand{\txind}[1]{#1\index{#1}}

\providecommand{\cl}[1]{\overline{#1}}
\providecommand{\x}{\times}

\providecommand{\com}[1]{}

\providecommand{\abs}[1]{\lvert#1\rvert}

\providecommand{\set}[1]{\lbrace#1\rbrace}

\providecommand{\bigset}[1]{\big\lbrace#1\big\rbrace}
\providecommand{\Bigset}[1]{\Big\lbrace#1\Big\rbrace}

\DeclareMathOperator{\supp}{supp}
\DeclareMathOperator{\id}{id}

\newcommand{\bd}{\partial}
\newcommand{\dd}[2]{\frac{\partial#1}{\partial#2}}
\newcommand{\ddsq}[2]{\frac{\partial^2#1}{{(\partial#2)}^2}}

\newcommand{\ddd}[3]{\frac{\partial^2#1}{\partial{#2}\partial{#3}}}

\providecommand{\dcd}{\,\cdot\,{,}\,\cdot\,}
\providecommand{\fdt}{\,\cdot\,}

\newcommand{\ind}{\chi\mathnormal{}}

\newcommand{\leb}{\lambda\hspace{-5 pt}\lambda}

\renewcommand{\phi}{\varphi}

\newenvironment{eqn*}{\begin{equation*}}{\end{equation*}}

\newcommand{\ie}{i.\,e.\@\xspace}
\newcommand{\eg}{e.\,g.\@\xspace}
\newcommand{\cf}{cf.\@\xspace}
\newcommand{\etc}{etc.\@\xspace}

\newcommand{\wrt}{w.\,r.\,t.\@\xspace}
\newcommand{\wlg}{w.\,l.\,o.\,g.\@\xspace}
\newcommand{\Wlg}{W.\,l.\,o.\,g.\@\xspace}
\newcommand{\resp}{resp.\@\xspace}

\providecommand{\WF}{Wright--Fisher\xspace}

\providecommand{\KBE}{Kolmogorov backward equation\xspace}
\providecommand{\KFE}{Kolmogorov forward equation\xspace}

\setkomafont{captionlabel}{\small\itshape}

\numberwithin{equation}{section} 
\swapnumbers
 
\theoremstyle{plain} 
\newtheorem{thm}{Theorem}[section]
\newtheorem{prop}[thm]{Proposition}
\newtheorem{lem}[thm]{Lemma}
\newtheorem{cor}[thm]{Corollary}

\theoremstyle{definition}
\newtheorem{dfi}[thm]{Definition}

\newtheorem{rmk}[thm]{Remark}

\allowdisplaybreaks[1]

\definecolor{gruen}{rgb}{0.6,0,0.5}

\newcommand{\blue}{\textcolor{blue}}

\renewcommand{\L}{\mathcal{L}}

\begin{document}

\title{A hierarchical extension scheme for backward solutions of the Wright–Fisher model}

\author{Julian Hofrichter, Tat Dat Tran, Jürgen Jost}

\date{\today}

\maketitle


\begin{abstract}
We develop an iterative global solution scheme for the backward Kolmogorov equation of the diffusion approximation of the Wright-Fisher model of population genetics. That model describes the random genetic drift of several alleles at the same locus in a population from a backward perspective. The key of our scheme is to connect the solutions before and after the loss of an allele. Whereas in an approach via stochastic processes or partial differential equations, such a loss of an allele leads to a boundary singularity, from a biological or geometric perspective, this is a natural process that can be analyzed in detail. A clarification of the role of the boundary resolves certain uniqueness issues and enlucidates the construction of hierarchical solutions. 
\end{abstract}


\section{Introduction}
The most basic mechanism of mathematical population genetics is
random genetic drift.  Parents are randomly chosen from the current
generation and transfer the alleles that they possess at some genetic
locus to their offspring.   The process is repeated over many generations, and once an
allele gets lost from the population because at the current step no
carrier of that allele is chosen as a parent, it will be lost forever
from the population. Thus, in the end, at each locus, only a single
allele will survive. This then leads to questions like the chances of
the different alleles present in the initial population to be the
survivor, or the expected times of the allele losses,  and so on. In
order to start an investigation of such questions, Fisher
\cite{fisher} and Wright \cite{wright1} developed the most basic
model. In that model, there is a finite population of finite size $N$
which is kept across the generations. Time is discrete, and each time
step, the parental generation is replaced by an offspring
generation. That offspring generation is formed by random sampling
with replacement. Mathematically speaking, this means that each of the
$N$ individuals in the offspring generation randomly and independently
chooses a parent. As each individual has only a single parent, no
recombination takes place. Each individual carries a single locus. At
this locus, each individual carries one of $n+1$ possible alleles labelled $0,1,\dots ,n$. 
Initially, the population possesses $n$ different alleles for that locus. There are no selective differences between those alleles, and no mutations occur. 

Of course, the model can be and has been generalized, to several loci, recombination, mutations, selective differences \etc, see \cite{ewens,buerger} for recent textbooks on mathematical population genetics. Nevertheless, the original model remains of considerable mathematical interest, and the issues around the loss of allele events contain some subtle mathematical structure. And that is what we shall focus upon in the present paper. Generalizations along the lines just indicated will then be presented elsewhere.

The mathematical investigation of the Wright-Fisher model owes much to
the pioneering work of Kimura \cite{kimura1,kimura2,kimura3}. A
crucial step was  not to work with the original model of a finite
population evolving in discrete time steps, but with the diffusion
approximation for an infinite population in continuous time. This then
leads to  the forward and backward Kolmogorov equations. The forward
equation is  a partial differential equation of parabolic type,
whereas  the backward equation, the  adjoint of the former \wrt a
suitable product, evolves backward in time and therefore is not
parabolic.  Mathematical difficulties arise from the fact that both
equations  become degenerate at the boundary.  In this paper, we shall
investigate the boundary behavior of the \KBE , that is
\bel{comp0}
-\dd{}{t} u(p,t) = \half\sum_{i,j=1}^n p^i(\delta^i_j-p^j)\ddd{}{p^i}{p^j}u(p,t)\ec L_n^\ast u(p,t),
\qe
where $p^i$ is the relative frequency of allele $i$; $p^0$ does not
appear in \rf{comp0} because of the normalization $\sum_{i=0}^n
p^i=1$. One readily sees that coefficients become 0 when one of the
frequencies $p^i$ becomes 0. Since we are working in the closure of
the probability
simplex $\Delta_n=\{(p^1,\dots ,p^n): p^i > 0, \sum_{j=1}^n p^j <
1\}$, this means that the PDE \rf{comp0} becomes degenerate at the
boundary of $\Delta_n$. (The fact that \rf{comp0} is not parabolic
because time is running backward is not such a serious problem,
because of the structure of the model and the duality with the --
parabolic -- \KFE .)

The Kolmogorov  equations have been studied with tools from the  theory of
stochastic processes, see for instance
\cite{ethier1,ethier2,ethier3,karlin}, and from the theory of  partial
differential equations  \cite{epstein1,epstein2}. These approaches,
because of their general nature, yield certain existence, uniqueness
and regularity results, but cannot come up with explicit formulas, for
instance for the expected time of loss of an allele. Therefore, other
authors focused on  the specific and explicit structure of the model.
Among many other things, the global aspect, that is, connecting the solutions in the
interior of the simplex and on its boundary faces, has been addressed
in the literature, and a number of representation formulas has been
derived. There is some discussion in Section 5.10 of \cite{ewens}, as
well as in \cite{buerger}, but we wish to describe some of the
relevant results in more detail and with a different focus. 

The first  solution schemes for the Kolmogorov equations were of a
local nature. In 1956, Kimura solved the \KFE for the 3-allelic case
($n=2$) in~\cite{kimura_3all}. Baxter, Blythe and McKane
in~\cite{BBMcK} solved the case of  an arbitrary number of alleles by
separation of variables. And in fact, the \KBE also always has simple
global stationary solutions  (\cf section~\ref{sec_long-term}). The
main achievement of this paper will be to compile the existing local
solutions into a non-trivial global solution by handling the  boundary
singularities. 

In the literature, using an observation of \cite{Sato1978}, one usually writes the Kolmogorov backward operator
in the form
\begin{align}\label{comp1}
 \Lambda_n^\ast u(x,t) \ce \half\sum_{i,j=0}^n x^i(\delta^i_j-x^j)\ddd{}{x^i}{x^j}u(x,t),
\end{align}
using the variables $(x^0,x^1,\dots ,x^n)$ with $\sum_{j=0}^n x^j=1$ 
in place of $L_n^\ast u(p,t)$ (\cf equation~\eqref{comp0}) 
with $(p^1,\dots ,p^n)$ and $p^0=1-\sum_{i=1}^n p^i$ implicitly
determined (for our notation, see Sections \ref{sec_simplex} and \ref{kolmogorov}, in particular \rf{eq_stand_simpl} and \rf{eq_Ln*_def}), that is, one works on the simplex $\{x^0 +x^1 +\dots x^n=1, x^i\ge
0\}$, i.e., the variable $x^0$ is included. This has the advantage of
being symmetric w.r.t. all $x^i$, but the disadvantage that the
operator invokes more independent variables than the dimension of the
space on which it is defined. In other words, the elliptic operator
becomes degenerate. In our treatment, we have opted to
work with $L_n^\ast$, but for the comparison with the literature, we shall
utilize the version \rf{comp1}. 

The starting point of much of the literature to be referenced here is
the observation of Wright \cite{wright2} that when one includes
mutation, the degeneracy at the boundary is removed. More precisely,
let the mutation rate $m_{ij}$ be the probability that when allele $i$ is selected for
offspring, the offspring carries the mutant $j$ instead of $i$. One
also puts $m_{ii}=-\sum_{j\ne i} m_{ij}$. The
corresponding Kolmogorov backward operator then becomes
\begin{align}\label{comp2}
 \Lambda_n^\ast u(x,t) \ce
 \half\sum_{i,j=0}^n x^i(\delta^i_j-x^j)\ddd{}{x^i}{x^j}u(x,t) + \sum_{j=0}^n \sum_{i=0}^n m_{ij}x^i \frac{\partial}{\partial x^j}.
\end{align}
Wright \cite{wright2} then discovered that a mathematically very
convenient assumption is 
\bel{comp3}
m_{ij}=\frac{1}{2}\mu_j >0 \text{ for }i\neq j,
\qe
that is,  the mutation rates only depend on the target gene (the
factor $\frac{1}{2}$ is inserted soly for purposes of normalization)
and are positive.  With \rf{comp3}, \rf{comp2} becomes
\begin{align}\label{comp4}
 \Lambda_n^\ast u(x,t) \ce
 \half\sum_{i,j=0}^nx^i(\delta^i_j-x^j)\ddd{}{x^i}{x^j}u(x,t) + \half \sum_{j=0}^n (\mu_j -\sum_{i=0}^n \mu_i)x^j \frac{\partial}{\partial x^j}.
\end{align}
In this case, the Wright-Fisher diffusion has a unique stationary
distribution, given by the Dirichlet distribution with parameters
$\mu_0,\dots ,\mu_n$. 
A further simplification occurs when
\bel{comp5}
\mu_0 = \dots =\mu_n =:\mu >0,
\qe 
that is, when all mutation rates are the
same. The assumption \rf{comp3} that the mutation rates only depend on
the target gene is not so natural biologically (the mutation rate
should rather depend on the initial instead of the target gene, but
\rf{comp5} remedies that deficit in a certain sense), but for our purposes
the more crucial issue is the assumption of positivity. 

Several papers have studied this model and derived explicit formulas
for the transition density of the process with generator \rf{comp4};
they include \cite{lit-fack,Shi1977,Gri1979,Gri1980,Shi1981,Tav1984,EG1993,GS2010}. A powerful tool in this line of research has
been Kingman's coalescent \cite{King1982}, that is, the method of
tracing lines of descent back into the past and analyzing their
merging patterns (for a quick introduction to
that theory, see also \cite{jost_bio}). In particular, some of these formulas
also apply in the limiting case $\mu=0$ in \rf{comp5}.
Ethier-Griffiths \cite{EG1993} showed that the following formula for the
transition density
\bel{lit1}
P(t,x,dy)=\sum_{M \ge 1} d^0_M(t) \sum_{|\alpha|=M, \alpha \in\Z^n_+}{|\alpha| \choose \alpha}x^{\alpha}\mathrm{Dir}(\alpha, dy),
\qe
which had earlier been derived under the assumption $\mu >0$, pertains to the case $\mu=0$. 
Here, $\mathrm{Dir}$ is the Dirichlet distribution, and $ d^0_M(t)$ is
the number of equivalence classes of lines of descent of length $M$ at time $t$ in Kingman's coalescent
for which analytical formulas have been derived in
\cite{Tav1984}. \rf{lit1} has been studied further in many subsequent
papers, for instance \cite{GS2010}. 
Shimakura \cite{Shi1981} has the less explicit formula
\ba
\nonumber
P(t,x,dy)&=&\sum_{m\ge 1} e^{-\lambda_m t} E_m(x,dy)\\
\nonumber
 &=& \sum_{K\in \Pi} P(t,x,y) dS_K(y)\\
\label{lit2}
 &=&\sum_{K\in \Pi} e^{-\lambda_m t} E_{m,K}(x,y)dS_K(y).
\ea
Here, the $\lambda_m$ are the eigenvalues introduced above, and 
$E_m$ stands for the projection onto the corresponding eigenspace, and
the index $K$ enumerates the faces of the simplex. 
The Dirichlet distribution in \rf{lit1} and the measure $dS_K(y)$ in
\rf{lit2} both
become singular when $y$ approaches the boundary of $K$. The point
here is that the sum invokes solutions on the individual faces, and
the transition from one face into one of its boundary faces becomes
singular in this scheme. In fact, \rf{lit2} is simply a decomposition
into the various modes of the solutions of a linear PDE, summed over
all faces of the simplex. 

In this paper, we want to get a more detailed analytical picture of
the behavior at the boundary and develop a global solution on the entire state space including its stratified boundary. In an important recent work, Epstein and
Mazzeo \cite{epstein1,epstein2} have developed PDE techniques to
address the issue of solving  PDEs on a manifold with corners that degenerate at the boundary  with the same leading terms as the  Kolmogorov backward equation for the Wright-Fisher model
\bel{comp10}
-\dd{}{t} u(p,t) + L^*_n u(p,t)=f,
\qe 
in the closure of the probability simplex \text{in $(\cl{\D}_n)_{-\infty}=\cl{\D}_n\times(-\infty,0)$} (see also lemma~\ref{lem_restr_bwd}). These results apply to a rather wide class of such PDEs. An important part of their work is the identification of appropriate function spaces. In our context, their spaces $C^{k,\gamma}_{WF}(\cl{\D}_n)$ would consist of $k$ times differentiable functions whose $k$th derivatives are Hölder continuous with exponent $\gamma$ w.r.t. the Fisher metric. (This only holds true for $L^*_n$, although E\&M also use this construction for their generalised setting.) In terms of the Euclidean metric on the simplex, this means that a weaker Hölder exponent (essentially $\frac{\gamma}{2}$) is required in the normal than in the tangential directions at the boundary. They then show that when the right hand side $f$ of \rf{comp10} is of class $C^{k,\gamma}_{WF}(\cl{\D}_n)$ for some $k\ge 0, 0< \gamma <1$, and if the initial values are of class  $C^{k,2+\gamma}_{WF}(\cl{\D}_n)$ (essentially $C^{k+2}$ with a suitable Hölder condition on $k$th, $(k+1)$th and scaled $(k+2)$th derivatives), then there exists a unique solution in that latter class. This result is very satisfactory from the perspective of PDE theory (see e.g. \cite{jost_pde}). Here, however, we are considering solutions that are not even continuous, let alone of some class $C^{0,2+\gamma}(\cl{\D}_n)$, as we want to study the boundary transitions (nevertheless, there are some points of contact in section~\ref{sec_long-term}). Therefore, in this paper, we carry out a
detailed investigation of the boundary behavior of solutions of
\rf{comp10}. A particular issue is the relation between several loss
of allele events that can occur in different possible orders. In
analytical terms, the issue is the regularity of solutions at
singularities of the boundary, that is, where two or more faces of the
simplex $\Delta_n$ meet. We also consider particular extension paths
from the boundary into the interior of the simplex. They have nothing
to do, however, with Kingman's coalescent lines of descent as utilized
in some of the literature discussed above. Kingman's scheme is
concerned with tracing common ancestors of members of the current
population of alleles. In contrast to that, we are interested in the
directions in which the singularities of the boundary of the simplex
are approached from the interior, because we are interested in the
continuity at the boundary.

In contrast to the approaches discussed above that invoke strong tools
from the theory of stochastic processes, our approach is not
stochastic, but analytic and geometric in nature. In that sense, our
approach is closer in spirit to that of \cite{epstein1,epstein2}. In
contrast to that approach, however, we develop  geometric constructions, within the framework of  information geometry, that is, the geometry of probability distributions, see \cite{amari,ajls}, in order to have an approach that on one hand is naturally capable of studying such generalizations as indicated above, but on the other hand can still derive explicit formulas. This is part of a general research program, see \cite{Dat,julian,THJ1,THJ2,THJ3,THJ4}. The present paper, which is based on \cite{julian}, is the backward counterpart to \cite{THJ4}, which investigated the \KFE.

\medskip

The solutions of the  backward Kolmogorov equation are the probability
distribution over ancestral states yielding some given current state
of allele frequencies. Thus, time runs backward, indeed, as the name
indicates. Such an ancestral state could have possessed more alleles
than the current state, because on the path towards that latter state,
some alleles that had been originally present in the population could
have been lost. In analytical terms, one could assume that such a loss
of allele event is continuous, in the sense that the relative
frequency of the corresponding allele simply goes to~0. Geometrically,
however, this means that the process moves from the interior of a
probability simplex into some boundary stratum and henceforth stays
there. Also, when two or more alleles got lost, they could have
disappeared in different orders from the population. The main
achievement of the current paper then is a global and hierarchical
solution for the \KBE  that persists and stays regular across
different such loss of allele events in the past. This is technically rather involved and has not been achieved before in the literature. For a complete understanding and a rigorous solution of the \KBE , however, this is indispensable. Of course, one needs to know how many alleles there had been in the original population, or equivalently, how many alleles got lost between the ancestral and the current state of the populations.

Here is a more precise description of the content of this paper.  We face the issue of  the degeneracy at the boundary of the Kolmogorov equations head on. Although creating analytical difficulties,  biologically and  geometrically, this is very natural because it corresponds to the loss by random drift of some alleles  from the population in finite time. As mentioned above, this has to happen almost surely. After an allele gets lost, the population keeps evolving by random genetic drift. The Wright-Fisher process has to be applied  with fewer alleles than before, but otherwise there is no conceptual difference. Of course, the process stops when  only one allele is left. Therefore, it is biologically essential and geometrically natural to connect the processes before and after the loss of an allele.

In the current backward setting, the perspective of an allele loss is reversed: A process with, say, $n$ alleles taking place on an $(n-1)$-dimensional probability simplex, may originate from a process with $n+1$ alleles on an $n$-dimensional simplex by loss of an allele. The former then should be identified as a facet of the latter, that is, the loss of an allele simply means that the process moves from the interior into the boundary of the simplex of subsequent higher dimension from. Of course, this will be repeated backwards, incorporating the previous loss of further alleles. Thus, the process could have originated from higher and higher dimensional simplices until its ultimate starting configuration. In this paper, we therefore construct a global solution that incorporates and connects these successive loss of allele events. In technical terms, we develop a hierarchical scheme that relies on a careful analysis of the connection modes and  tailored regularity specifications for the corresponding solutions.

\subsection*{Acknowledgements}
The research leading to these results has received funding from the European Research Council under the European Union's Seventh Framework Programme (FP7/2007-2013) / ERC grant agreement n$^\circ$~267087. J.\,H. and T.\,D.\,T. have also been supported by scholarships from the IMPRS ``Mathematics in the Sciences'' during earlier stages of this work.

\section{Preliminaries and notation}\label{sec_simplex}

We consider a population that initially carries $n+1$ different alleles at a single locus. The allele distribution of the next generation is chosen by random sampling with replacement from the current generation. In other words, we repeatedly sample a binomial distribution. 
As pioneered by Kimura, we consider the diffusion approximation of the process, where we let the population size $N\to \infty$ and rescale the discrete generation time as $t=\frac{1}{N}$. This leads to the Kolmogorov equations for the  evolution of the probability distribution of the alleles. In contrast to the population size, the number of alleles is kept finite. Therefore, as we are interested in the relative allele frequencies, the state space is the $n$-dimensional probability simplex.

In this section, we shall recall the notation from \cite{THJ4} that is necessary for the iterative transition to boundary strata of this simplex within a hierarchical scheme, as well as the appropriate function spaces. 

$p^0,p^1,\dots ,p^n$ denote the relative frequencies of alleles $0,1,\dots,n$. Because of $\sum_{j=0}^n p^j=1$, we   have $p^0=1-\sum_{i=1}^n p^i$. 
\begin{align}\label{eq_stand_simpl}
\D_n\ce\Bigset{{(p^1,\dotsc,p^n)\in\R^n\big\vert p^i > 0\text{ for $i=1,\dotsc,n$ and }\sum_{i=1}^n p^i < 1}},
\end{align}
is the  (open) \textit{$n$-dimen\-sional standard orthogonal simplex}\sindex[not]{Dn@$\D_n$}.  
Equivalently, we have 
\begin{align}
\D_n =\Bigset{{(p^0,\dotsc,p^n)\in\R^{n+1}\big\vert p^j > 0\text{ for }j=0,1,\dotsc,n \text{ and }\sum_{j=0}^n p^j=1}}.
\end{align}
The  topological closure of this simplex is 
\begin{align}\label{eq_simpl_I_n}
\cl{\Delta}_{n}=\bigset{{(p^1,\dotsc,p^n)\in\R^n\big\vert p^i \ge 0 \text{ for }i=1,\dotsc,n \text{ and }\sum_{i=1}^n p^i \le 1}}.
\end{align}
Time is $t\in (-\infty,0]$, and 
\begin{equation*}
  (\D_n)_{-\infty}:=\D_n\times(-\infty,0). 
\end{equation*}

The  subsimplices in the boundary $\bd\D_n=\cl{\D}_n\smin\D_n$ are called \textit{faces}, from the $(n-1)$-dimen\-sional \textit{facets} down to the 0-dimensional \textit{vertices}. Each subsimplex of dimension $k\leq n-1$ is isomorphic to the $k$-dimen\-sional standard orthogonal simplex $\D_k$. For an index set $I_k=\{i_0,i_1,\dots ,i_k\}\subset \set{0,\dotsc,n}$ with $i_j\neq i_l$ for $j\neq l$,
we put 
\begin{align} 
\Delta_{k}^{(I_{k})}\ce\Bigset{{(p^1,\dotsc,p^n)\in{\overline{\Delta}_n}\big\vert p^i > 0\text{ for $i\in I_k$; }p^i=0\text{ for $i\in I_n\smin I_k$}}}.
\end{align}
In particular,  $\D_n =\Delta_{n}^{(I_{n})}$.

Each of the $\binom{n+1}{k+1}$ subsets $I_k$ of $I_n$  corresponds to a  boundary face $\D_k^{(I_k)}$ ($k\leq n-1$). The \textit{$k$-dimen\-sional part of the  boundary $\bd_k\D_n$ of $\D_n$}\sindex[not]{dkDn@$\bd_k\D_n$} is therefore
\begin{align}\label{eq_bd_k}
\bd_k\D_n^{(I_{n})}\ce \bigcup_{I_k\subset I_n}\Delta_k^{(I_k)}\subset \bd\D_n^{(I_{n})}\quad\text{for $0\leq k\leq n-1$}.
\end{align}
Formal consistency thus also  leads to  $\bd_n\D_n =\D_n$. This boundary concept can be iteratively applied  to simplices in the boundary of some $\D_l^{(I_l)}$, $I_l\subset I_n$ for $0\leq k < l\leq n$. This means that 
\begin{align}
\bd_k\D_l^{(I_l)}=\bigcup_{I_k\subset I_l}\Delta_k^{(I_k)}\subset \bd\D_l^{(I_l)}.
\end{align}

The simplex  $\D_k^{(\set{i_0,\dotsc,i_{k}})}$  represents  the state where the $k+1$  the alleles $i_0,\dotsc,i_{k}$ are  present in the population, and $\bd_k\D_n$, that is, the union of all those simplices,  represents the state where the number of alleles is  $k+1$, but where their identity does not matter. When any one of the alleles   $i_0,\dotsc,i_{k}$ is  eliminated, we land in  $\bd_{k-1}\D_k^{(\set{i_0,\dotsc,i_{k}})}$.

We next introduce spaces of square integrable functions for our subsequent integral  products on $\D_n$ and its faces (which will be used implicitly, for details \cf \cite{THJ2}),
\begin{multline}\label{eq_dfi_L2_union}
L^2\Big(\bigcup_{k=0}^n\bd_k\D_n\Big)
\ce\Big\{f\co\cl{\D}_n\too\R\,\Big\vert\,\text{$f\vert_{\bd_k\D_n}$ is $\leb_k$-measurable and}\\\text{$\int_{\bd_k\D_n} \abs{f(p)}^2\,\leb_k(dp) < \infty$ for all $k=0,\dotsc,n$}\Big\}.
\end{multline}
Here, $\leb_k$ stands for the $k$-dimensional Lebesgue measure, but when integrating over some ${\Delta_k^{(I_k)}}$ with $0\notin I_k$, the measure needs to be replaced with the one induced on ${\Delta_k^{(I_k)}}$ by the Lebesgue measure of the containing $\R^{k+1}$ -- this measure, however, will still be denoted by $\leb_k$\label{pag_leb_k} as it is clear from the domain of integration ${\Delta_k^{(I_k)}}$ with either $0\in I_k$ or $0\notin I_k$ which version is actually used. In particular,  for the top-dimen\-sional simplex, we simply have 
\sindex[not]{L2Dn@$L^2(\D_n)$}
\begin{align}
L^2(\D_n)\ce\Big\{f\co\D_n\too\R\,\Big\vert\,\text{$f$ is $\leb_n$-measurable and $\int_{\D_n}\abs{f(p)}^2\,\leb_n(dp) < \infty$}\Big\}.
\end{align}

We also need spaces of $k$ times continuously differentiable functions,  for $k\in\N\cup\set{\infty}$, 
\sindex[not]{C0@$C_0^k(\cl{\Delta}_n)$}\sindex[not]{Cc@$C_c^k(\cl{\Delta}_{n})$}
\begin{align}
C_0^k(\cl{\Delta}_{n})&\ce\bigset{f\in C^k(\cl{\Delta}_{n})\big|f\vert_{\bd\D_n}=0},\\
C_0^k({\Delta}_{n})&\ce\bigset{f\in C^k({\Delta}_{n})
\big|\exists\,\bar{f}\in C_0^k(\cl{\Delta}_{n})\text{ with }\bar{f}\vert_{\D_n}=f}\\[-0.7em]
\intertext{as well as}\notag\\[-2.7em]
C_c^k(\cl{\Delta}_{n})&\ce\bigset{f\in C^k(\cl{\Delta}_{n})\big|\supp(f)\subsetneq{\Delta}_{n}}\label{eq_def_Cc}.
\end{align}

In order to define an extended solution on $\D_n$ and its faces (indicated by a capitalised $U$), we shall in addition need appropriate spaces of pathwise regular functions.  Such a solution needs to be at least of class $C^2$  in every boundary instance (actually, a solution typically always is of class $C^\infty$, which likewise applies to each boundary instance). Moreover, it should stay regular at boundary transitions that reduce the  dimension by one, \ie for $\D_k^{(I_k)}$ and a boundary face $\D_{k-1}\subset\bd_{k-1}\D_k^{(I_k)}$. Globally, we may require that such a property applies to all possible boundary transitions within $\cl{\D}_n$ and define correspondingly for $l\in \N\cup \set{\infty}$ \sindex[not]{Cplcl@$C_p^l\big(\cl{\D}_n\big)$}
\begin{align}\label{eql_reg_pathwise}
U\in C_p^l\big(\cl{\D}_n\big)\,\ceq\,{U}\vert_{\D_d^{(I_d)}\cup \bd_{d-1}\D_d^{(I_{d})}} \in C^l(\D_d^{(I_d)} \cup \bd_{d-1}\D_d^{(I_d)})\quad
\text{for all $I_d\subset I_n$, $1\leq d \leq n$}
\end{align}
with respect to the spatial variables. Likewise, for ascending chains of (sub-)simplices with a more specific boundary condition, we put for index sets $I_k\subset\dotsc\subset I_n$ and again for $l\in \N\cup\set{\infty}$
\sindex[not]{Cplcup@$C_{p_0}^l\Big(\bigcup_{d=k}^n\D_d^{(I_d)}\Big)$}
\begin{multline}\label{eql_reg_ext}
U\in C_{p_0}^l\Big(\bigcup_{d=k}^n\D_d^{(I_d)}\Big) \,\ceq\, 
\begin{cases}
U\vert_{\D_d^{(I_d)}}\text{ is extendable to }
\bar{U}\in C^l(\D_d^{(I_d)} \cup \bd_{d-1}\D_d^{(I_d)})\text{ with}\\
\bar{U}|_{\bd_{d-1}\D_d^{(I_d)}}=U\ind_{\D_{d-1}^{(I_{d-1})}}\ind_{\set{d>k}}\text{ for all $\max(1,k)\leq d \leq n$}
\end{cases}\hspace*{-16pt}
\end{multline}
with respect to the spatial variables. We note that such a function may straightforwardly be completed into a function defined on the entire $\cl{\D}_n$ by putting $U\ce0$ on $\cl{\D}_n\smin\big(\bigcup_{d=k}^n\D_d^{(I_d)}\big)$ for corresponding $t$; however, such an extension is generally not of class $C_p^l\big(\cl{\D}_n\big)$ \wrt the spatial variables.

\section{The Kolmogorov operators}\label{kolmogorov}

On an interior simplex $\D_n$, the \textit{\KBE} for the diffusion approximation of an $n$-allelic 1-locus \WF model reads \sindex[not]{Dninfty@$(\D_n)_\infty$}
\begin{equation}\label{eq_back_n}
\begin{cases}
-\dd{}{t} u(p,t) = L^*_n u(p,t)	&\text{in $(\D_n)_{-\infty}=\D_n\times(-\infty,0)$}\\
u (p,0) = f(p)			&\text{in $\D_n$, $f\in \L^2(\D_n)$}\\
\end{cases}
\end{equation}
for $u(\,\cdot\,,t)\in C^2(\D_n)$ for each fixed $t\in(-\infty,0)$ and $u(p,\,\cdot\,)\in C^1((-\infty,0))$ for each fixed $p\in\D_n$  and with the \textit{backward operator}\sindex[not]{Ln*@$L_n^*$}
\begin{align}\label{eq_Ln*_def}
L_n^* u(p,t) \ce \half\sum_{i,j=1}^n\big(p^i(\delta^i_j-p^j)\big)\ddd{}{p^i}{p^j}u(p,t).
\end{align}
Analogously, we have
\begin{align}\label{eq_Ln_def}
 L_n u(p,t) \ce \half\sum_{i,j=1}^n\ddd{}{p^i}{p^j}\big(p^i(\delta^i_j-p^j)u(p,t)\big)
\end{align}
being the \textit{forward operator}\sindex[not]{Ln@$L_n$}  
appearing in the corresponding \KFE. The definitions of the operators given in equations~\eqref{eq_Ln_def} and~\eqref{eq_Ln*_def} also apply to the closure $\cl{\D}_n$; we point this out as we shall also consider extensions of the solution and the differential equation to the boundary.

For relations between the two operators, we immediately have the following lemmas; the corresponding proofs may be found in \cite{THJ4}:

\begin{lem}\label{lem_adjoint}
$L_n$ and $L^*_n$ are (formal) adjoints with respect to the product $(\dcd)_n$ in the sense that
\begin{align}
(L_n u, \phi)_{n}= (u,L^*_n \phi)_n\quad\text{for $u\in C^2(\cl{\Delta}_n)$, $\phi\in C^2_0(\cl{\Delta}_n)$.}
\end{align}
\end{lem}

\begin{lem}\label{lem_ef-shift}\sindex[not]{on@$\omega_n$}
For an eigenfunction $\phi\in C^2(\cl{\Delta}_n)$ of $L_n$ and $\omega_n\ce\prod^n_{k=1}p^k\big(1-\sum_{l=1}^n p^l\big)$, we have: $\omega_n\phi\in C^2_0(\cl{\Delta}_n)$ is an eigenfunction of $L^*_n$ corresponding to the same eigenvalue and conversely.
\end{lem}

We continue with some further observations on the operators, in particular with regard to the  boundary of $\D_n$: The operator $L_n^*$, if restricted to subsimplices $\D_k^{(I_{k})}\cong\D_k$ in $\bd{\Delta^{(I_n)}_n}$ of any dimension~$k$, then again is the adjoint of the differential operator $L_k$ corresponding to the evolution of a $(k+1)$-allelic process in $\Delta_k$:

\begin{lem}\label{lem_restr_bwd}
For $0\leq k < n$ and $I_k\subset\set{0,\dotsc,n}$, $\abs{I_k}=k$, we have
\begin{align}
L_n^*\big\vert_{\D_k^{(I_{k})}}=L_k^*.
\end{align}
\end{lem}
We may therefore omit the index~$k$ in $L_k^*$ whenever convenient, in particular when considering domains where (parts of) the boundary are included. For the  operator $L_n$, we do not have such a restriction property\label{pag_op_restr_fwd}.


The probabilistic interpretation is that the backward solution $u(p,t)$ \sindex[not]{u(p,t)@$u(p,t)$} expresses the probability of having started in some $p\in\D_n$ at the negative time~$t$ conditional upon being in a certain state $u(p,0)=f(p)$ at time $t=0$, \ie having reached the corresponding (generalised) \txind{target set}.

\section{Solution schemes for the \KBE}

Solutions of the \KBE and of the \KFE are linked by the adjointness relation for the Kolmogorov operators $L_n$ and $L_n^*$ given in lemma~\ref{lem_adjoint}, and hence known solution schemes (\cf \cite{kimura_3all}, \cite{BBMcK}) are essentially applicable for either equation.  However, there is a subtle difference in the context of the non-matching spectra of $L_n$ and $L_n^*$ (\cf \cite{THJ2}): All eigenfunctions of $L^*$ acquired by the adjointness relation in lemma~\ref{lem_ef-shift} are in $C_0^\infty(\D_n)$, but $L_n^*$ in $\D_n$ possesses even more eigenfunctions (in particular for smaller eigenvalues) since all eigenfunctions of $L_k^*$ in $\D_k$ for some $0\leq k < n$ also occur as eigenfunctions of $L_n^*$ by \eg constant extension. 

With the eigenfunctions (\eg the generalised Gegenbauer polynomials, \cf \cite{Dat}) given, the construction of a solution of equation~\eqref{eq_back_n} in $\D_n$ is rather straightforward. However, the -- in comparison with the forward case -- larger set of eigenfunctions causes ambiguities when decomposing a final condition, which prevents uniqueness results for the solution. But if we restrict the choice of eigenfunctions to the `proper' eigenfunctions in the domain%
\footnote{This is also sufficient as their linear span is already dense in $C_c^\infty(\cl{\D}_n)$ and consequently also in $\L^2(\D_n)$: Linear combinations of the generalised Gegenbauer polynomials as eigenfunctions of $L_n$ (\cf \cite{Dat}) are dense in $C^\infty(\cl{\D}_n)$;  dividing a function $f\in C_c^\infty(\cl{\D}_n)$ by $\omega_n$ (\cf lemma~\ref{lem_ef-shift}) again yields a function in $C_c^\infty(\cl{\D}_n)\subset C_0^\infty(\cl{\D}_n)$ as $\omega_n$ is in $C_0^\infty(\cl{\D}_n)$ itself and positive in the interior~${\Delta}_{n}$.}%
, \ie those in $C_0^\infty(\D_n)$, which are derived from eigenfunctions of $L_n$,
the existence and uniqueness of a solution observed in the forward case likewise apply. Thus, for such a solution 
by proper eigenfunctions (which will be called a \textit{proper solution of the \KBE in $\D_n$}), we have, coinciding with the result of \eg \cite{lit-fack}:

%

\begin{prop}\label{prop_sol_back_n}
For $n\in\N$ and a given final condition  $f\in \L^2(\D_n)$, the \KBE corresponding to the diffusion approximation of the $n$-dimen\-sional \WF model  \eqref{eq_back_n} always allows a unique proper solution $u\co{\big({\Delta}_{n}\big)}_{-\infty}\map\R$ 
with $u(\,\cdot\,,t)\in C_0^\infty(\D_n)$ for each fixed $t\in(-\infty,0)$ and $u(p,\,\cdot\,)\in C^\infty((-\infty,0))$ for each fixed $p\in\Delta_{n}$. 
\end{prop}

By construction, proper solutions do not cover the boundary. In the next section, the non-proper components will be interpreted as originating from (proper) solutions on lower-dimensional boundary strata. 

\section{Inclusion of the boundary and the extended \KBE}\label{sec_hier_sol_bkw}

We shall now  include the boundary and  its contribution into the model.  
%
%
We  augment the domain of equation~\eqref{eq_back_n} such that it comprises the entire $\cl{\D}_n$ 
yielding what we call the \textit{extended \KBE}
\begin{align}\label{eq_back_n_ext}
\begin{cases}
-\dd{}{t} U(p,t)=L^* U(p,t) 	&\text{in ${\big(\overline{\Delta}_{n}\big)}_{-\infty}=\cl{\D}_n\times(-\infty,0)$}\\%
U (p,0) =f(p)					&\text{in $\overline{\Delta}_{n}$, $f\in \L^2\big(\bigcup_{k=0}^n\bd_k\D_n\big)$}\\
\end{cases}
\end{align}
for $U(\,\cdot\,,t)\in C_p^2\big(\cl{\D}_n\big)$ 
for each fixed $t\in(-\infty,0)$ and $U(p,\,\cdot\,)\in C^1((-\infty,0))$ for each fixed $p\in\overline{\Delta}_{n}$. Here,  $f$\sindex[not]{f@$f$} is the
\textit{extended final condition}  which is defined on $\overline{\Delta}_{n}$. Thus,  any boundary instance of the boundary of the simplex may also belong to the \txind{target set} considered.


Our problem now is different from standard final-boundary value problems, because
for such a solution, the configuration on the boundary is no longer static in general, but is governed by $L_k^*$ with~$k$ being the corresponding dimension \resp  by $L_n^*$ restricted to the corresponding domain, matching the degeneracy behaviour of $L_n^*$ (\cf lemma~\ref{lem_restr_bwd}).  Hence, the index may be omitted, and 
we may just write $L^*$ (for dimension 0, we formally put $L^* = L^*_0\ce0$ there). In terms of the underlying \WF model,  this signifies that the boundary is subject to the same type of evolution, merely in  a different dimension, justifying the choice of equation~\eqref{eq_back_n}.

The key point now is to connect the different boundary strata,  by requiring $U\in C^2_p(\overline{\D}_n)$ \wrt the spatial variables: Clearly, inside each boundary instance the solution needs to be sufficiently regular for $L^*$, but regarding the boundary, we also demand such regularity for simple boundary transitions, \ie when  the  dimension decreases by one. For higher order transitions, however, irregularities are admitted. This corresponds to the degeneracy behaviour of the operator at the boundary and will be observed with the solutions constructed. This allows for a much wider class of global solutions. These solutions are not artificial, but  correspond to natural scenarios in the underlying \WF model.

\section{An extension scheme for solutions of the \KBE}

We want to construct the class of global solutions of the \KBE~\eqref{eq_back_n} by successive backward extension of local solutions in different boundary strata. For this, we  first look at single extensions of solutions from a boundary instance of the considered domain to the interior. The extensions  are confined by: 


\begin{dfi}[extension constraints]\label{dfi_ext}
Let $I_d$ be an index set with $\abs{I_d}=d+1\geq 2$, $0,s\in I_d$ and $\D_d^{(I_d)}=\set{(p^i)_{i\in I_d\setminus\set{0}}\vert p^i>0\text{ for $i\in I_d$}}$  with $p^0\ce 1-\sum_{i\in I_d\setminus\set{0}}p^i$. For $d\geq2$ and a solution $u\co \big({\D_{d-1}^{(I_d\setminus\set{s})}}\big)_{-\infty}\too \R$ of the \KBE~\eqref{eq_back_n}, 
\ie ${u}(\fdt,t)\in C^\infty\big(\D_{d-1}^{(I_d\smin\set{s})}\big)$ for $t<0$,  $u(p,\,\cdot\,)\in C^\infty((-\infty,0))$ for $p\in\D_{d-1}^{(I_d\smin\set{s})}$ and
\begin{align}
-\dd{}{t}u= L^*u
\quad\text{in $\big({\D_{d-1}^{(I_d\smin\set{s})}}\big)_{-\infty}$},
\end{align}
a function $\bar{u}\co \big({\D_{d}^{(I_d)}}\big)_{-\infty}\too \R$ with $\bar{u}(\fdt,t)\in C^\infty\big(\D_d^{(I_d)}\big)$ for $t<0$ and $\bar{u}(p,\,\cdot\,)\in C^\infty((-\infty,0))$ for $p\in\D_{d}^{(I_d)}$
is said to be an extension of $u$ in accordance with the \textit{extension constraints} if
\begin{itemize}
 \item[(i)]{for~$t<0$ $\bar{u}(\fdt,t)$ is continuously extendable to the boundary $\bd_{d-1}\D_d^{(I_d)}$ such that it coincides with $u(\fdt,t)$ in $\D_{d-1}^{(I_d\setminus\set{s})}$ \resp vanishes on the remainder of $\bd_{d-1}\D_d^{(I_d)}$ and is of class $C^\infty$ with respect to the spatial variables in $\D_d^{(I_d)}\cup\bd_{d-1}\D_d^{(I_d)}$,}
 \item[(ii)]{it is a solution of the corresponding \KBE in $\big({\D_d^{(I_d)}}\big)_{\hspace*{-2pt}-\infty}$\hspace*{-2pt}, \ie $-\dd{}{t}\bar{u}= L^*\bar{u}$ in $\big({\D_d^{(I_d)}}\big)_{-\infty}$.}
\end{itemize}
For $d=1$, this analogously applies to functions $u$ with $-\dd{}{t}u=0$ (in accordance with  $L^*_0\equiv 0$), and consequently the equation in condition~(ii) is replaced with $L^*\bar{u}=0$. 
Furthermore, an extension which encompasses multiple extension steps is said to be in accordance with the {extension constraints}, if this holds for every extension step.
\end{dfi}

\begin{rmk}\label{rmk_ext}
In case of $d\geq2$, if $u$ for $t<0$  extends smoothly to the boundary $\bd_{d-2}\D_{d-1}^{(I_d\smin\set{s})}$ such that this extension vanishes everywhere on $\bd_{d-2}\D_{d-1}^{(I_d\smin\set{s})}$, 
the above definition corresponds to $(u\ind_{\D_{d-1}^{(I_d\smin\set{s})}}+\bar{u}\ind_{\D_{d}^{(I_{d})}})
\in C^\infty_{p_0}(\D_{d-1}^{(I_d\smin\set{s})}\cup\D_d^{(I_d)})$ with respect to the spatial variables for~$t<0$ (\cf equality~\eqref{eql_reg_ext}) except for the \KBE solution property.
\end{rmk}

We shall  investigate here  the existence of such extensions which comply with definition~\ref{dfi_ext}; the issue of their uniqueness will be dealt with in another paper. Corresponding to the chosen separation ansatz (on which the  result~\ref{prop_sol_back_n} is based), we shall have to construct  extensions  of the eigenmodes:

\begin{lem}[extension of eigenfunctions]\label{lem_extension}
Let $I_d$ be an index set with $\abs{I_d}=d+1\geq 2$, $0,s\in I_d$ and $\D_d^{(I_d)}=\set{(p^i)_{i\in I_d\setminus\set{0}}\vert p^i>0\text{ for $i\in I_d$}}$  with $p^0\ce 1-\sum_{i\in I_d\setminus\set{0}}p^i$. For $d\geq2$ and   
an eigenfunction $\psi\in C^\infty\big(\D_{d-1}^{(I_d\setminus\set{s})}\big)$ of $L_{d-1}^\ast$ for the eigenvalue $\kappa\geq 0$, \ie 
\begin{align}
L_{d-1}^*\psi=-\kappa\psi\quad\text{ in $\D_{d-1}^{(I_d\setminus\set{s})}\subset\bd\D_d^{(I_d)}$},
\end{align}
a linear interpolation $\bar{\psi}=\bar{\psi}^{r,s}\co\D_d^{(I_d)}\map\R$ of $\psi$ from $\D_{d-1}^{(I_d\setminus\set{s})}$ (source face) 
towards $\D_{d-1}^{(I_d\setminus\set{r})}\subset\bd_{d-1}\D_d^{(I_d)}$ for some $r\in I_d\smin\set{s}$ (target face) 
is given by \sindex[not]{prs@$\pi^{r,s}$}
\begin{equation}\label{eq_extension}
\bar{\psi}^{r,s}(p)\ce \psi(\pi^{r,s}(p)) \cdot\frac{p^{r}}{p^s+p^r}\quad\text{for $p\in\D_d^{(I_d)}$}
\end{equation}
with $\pi^{r,s}(p^1,\dotsc,p^d)=(\tilde{p}^1,\dotsc,\tilde{p}^d)$
such that $\tilde{p}^s=0$, $\tilde{p}^r=p^s+p^r$ and $\tilde{p}^i=p^i$ for $i\in I_d\smin\set{s,l}$.


The regularity of  
$\bar{\psi}$ 
 corresponds to that of  $\psi$ in $\D_d^{(I_d)}$ (\ie it is of class $C^\infty$) and satifies 
\begin{align}
&L_d^*\bar{\psi}=-\kappa\bar{\psi}\quad\text{in $\D_d^{(I_d)}$}.\label{eq_homsol-ext}
\end{align}
Moreover, $\bar{\psi}$ extends smoothly to $\D_{d-1}^{(I_d\setminus\set{s})}$ and $\D_{d-1}^{(I_d\setminus\set{r})}$, and there we have
\begin{align}
&\bar{\psi}|_{\D_{d-1}^{(I_d\setminus\set{s})}}=\psi,\qquad
\bar{\psi}|_{\D_{d-1}^{(I_d\setminus\set{r})}}=0.
\end{align}
If furthermore $\psi$ extends smoothly to $\D_{d-2}^{(I_d\setminus\set{s,q})}\subset{\bd_{d-2}\D_d^{(I_d\setminus\set{s})}}$ for some $q\in I_d\setminus\set{r,s}$,  
then $\bar{\psi}$ likewise extends smoothly to $\D_{d-1}^{(I_d\setminus\set{q})}$.
In particular, $\bar\psi$ satifies the extension constraint~\ref{dfi_ext}\,(i) if $\psi$ extends smoothly to $\bd_{d-2}\D_{d-1}^{(I_d\setminus\set{s})}\smin\D_{d-2}^{(I_d\setminus\set{r,s})}$ and vanishes there. 
 
For $d=1$, the preceding statements analogously hold for arbitrary $\psi\co \D_{0}^{(I_1\setminus\set{s})}\map\R$ as eigenfunction of $L^*_0\equiv 0$ for the eigenvalue $0$; then, 
$\bar{\psi}$ is of class $C^\infty$ in $\D_1^{(I_1)}$, and such an extension is always in accordance with the extension constraint~\ref{dfi_ext}\,(i).

\end{lem}

Since the eigenfunctions are the building blocks for a solution scheme, the preceding lemma directly extends to solutions of the \KBE:


\begin{prop}[extension of solutions]\label{prop_extension}
Let $I_d$ be an index set with $\abs{I_d}=d+1\geq 2$, $0,s\in I_d$ and $\D_d^{(I_d)}=\set{(p^i)_{i\in I_d\setminus\set{0}}\vert p^i>0\text{ for $i\in I_d$}}$  with $p^0\ce 1-\sum_{i\in I_d\setminus\set{0}}p^i$.
For $d\geq2$,  a given \txind{final condition} $f\in\L^2\big({\D_{d-1}^{(I_d\setminus\set{s})}}\big)$ and a given extension target face index $r\in I_d\setminus\set{s}$, a solution $u\co \big({\D_{d-1}^{(I_d\setminus\set{s})}}\big)_{-\infty}\too \R$ of the  \KBE~\eqref{eq_back_n}, ${u}(\fdt,t)\in C^\infty\big(\D_{d-1}^{(I_d\smin\set{s})}\big)$ for $t<0$ and $u(p,\,\cdot\,)\in C^\infty((-\infty,0))$ for $p\in\D_{d-1}^{(I_d\smin\set{s})}$, may be extended to a function \sindex[not]{urs@$\bar{u}^{r,s}$}
\begin{align}
\bar{u}=\bar{u}^{r,s}\co \big({\D_d^{(I_d)}}\big)_{-\infty}\too \R
\end{align}
with $\bar{u}(\fdt,t)\in C^\infty\big(\D_d^{(I_d)}\big)$ for $t<0$ and $\bar{u}(p,\,\cdot\,)\in C^\infty((-\infty,0))$ for $p\in\D_{d}^{(I_d)}$ as well as satisfying
\begin{align}\label{eq_KBE_ubar}
-\dd{}{t}\bar{u}= L^*\bar{u}
 \quad\text{in $\big({\D_d^{(I_d)}}\big)_{-\infty}$}.
\end{align}
Furthermore, for~$t<0$ $\bar{u}(\fdt,t)$ smoothly extends to the boundary in ${\D_{d-1}^{(I_d\setminus\set{s})}}$  with
\begin{align}
\bar{u}(\fdt,t)|_{\D_{d-1}^{(I_d\setminus\set{s})}}=u,\quad\text{in particular}\quad\bar{u}(\fdt,0)|_{\D_{d-1}^{(I_d\setminus\set{s})}}=f\vert_{\D_{d-1}^{(I_d\setminus\set{s})}}
\end{align}
and in ${\D_{d-1}^{(I_d\setminus\set{r})}}$ with $\bar{u}(\fdt,t)|_{\D_{d-1}^{(I_d\setminus\set{r})}}=0$.
If furthermore $u(\fdt,t)$ for $q\in I_d\setminus\set{r,s}$ extends smoothly to $\D_{d-2}^{(I_d\setminus\set{q,s})}\subset{\bd_{d-2}\D_d^{(I_d\setminus\set{s})}}$ for some~$t$,
then $\bar{u}(\fdt,t)$ likewise extends smoothly to $\D_{d-1}^{(I_d\setminus\set{q})}$. In particular, $\bar{u}$ satifies the extension constraints~\ref{dfi_ext} if $u(\fdt,t)$ extends smoothly to $\bd_{d-2}\D_{d-1}^{(I_d\setminus\set{s})}\smin\D_{d-2}^{(I_d\setminus\set{r,s})}$ and vanishes there for~$t<0$.

For $d=1$, the preceding analogously holds for functions $u\co \big(\D_{0}^{(I_1\setminus\set{s})}\big)_{-\infty}\map\R$ with $u(p,\cdot\,)\in C^\infty((-\infty,0))$ and $\dd{}{t}{u}=0$; then, $\bar{u}(\fdt,t)$ is of class $C^\infty$ in $\D_1^{(I_1)}$ for every $t$ as well as $\bar{u}(p,\,\cdot\,)\in C^\infty((-\infty,0))$ for $p\in\D_{d}^{(I_d)}$ with $\dd{}{t}{\bar{u}}=0$, and equation~\eqref{eq_KBE_ubar} holds correspondingly.
Furthermore, this extension always is in accordance with the extension constraints~\ref{dfi_ext}.

\begin{rmk}\label{rmk_ext_t_0}
The extension of a solution of the \KBE for a final condition $f\in\L^2\big({\D_{d-1}^{(I_d\setminus\set{s})}}\big)$ as in proposition~\ref{prop_extension} is also applicable for $t=0$, yielding an analogously extended final condition \sindex[not]{frs@$\bar{f}^{r,s}$} $\bar{f}=\bar{f}^{r,s}\in \L^2\big({\D_d^{(I_d)}}\big)$. We then have  $\bar{u}(\fdt,0)\equiv\bar{f}$ in ${\D_d^{(I_d)}}$ by continuous extension as we have ${u}(\fdt,0)=f$ in $\D_{d-1}^{(I_d\setminus\set{s})}$;  
however, for $d\geq2$ this extension of $f$ in general does not have the boundary regularity described due to the missing regularity of $f$ (and hence in general does not satisfy the extension boundary constraint~\ref{dfi_ext}\,(i)).
\end{rmk}


\end{prop}

In addition to the preceding proposition, it should be noted that $\bar{u}$ does not necessarily extend continuously to the entire $\cl{\D}_{d}$, in particular not to the remaining boundary parts of dimension $d-2$ and less. This is due to the fact that on instances of $\bd_{d-2}\D_d^{(I_d)}$, which are shared boundaries of higher-dimen\-sional faces of the simplex, continuous extensions from each of those faces  may exist, but do not necessarily coincide. 

\begin{proof}[Proof of lemma~\ref{lem_extension}]
The regularity assertion for $\bar{\psi}$ in $\D_d^{(I_d)}$ follows from the regularity of $\pi$ and of the projection and from $\frac{p^{r}}{p^s+p^r}$ being of class $C^\infty$ on $\D_d^{(I_d)}$. The boundary behaviour is similarly straightforward as $\pi^{r,s}=\id$ and $\frac{p^{r}}{p^s+p^r}=1$ on $\D_{d-1}^{(I_d\setminus\set{s})}$, whereas $\frac{p^{r}}{p^s+p^r}=0$ on $\D_{d-1}^{(I_d\setminus\set{r})}$. Both boundary extensions are smooth in the sense described, which is again due to the regularity of the projection and of $\frac{p^{r}}{p^s+p^r}$ when approaching $\D_{d-1}^{(I_d\setminus\set{s})}$ \resp $\D_{d-1}^{(I_d\setminus\set{r})}$. 
Analogous considerations yield the assertion for other boundary faces of $\bd_{d-1}\D_d^{(I_d)}$: The projection $\pi^{r,s}$ maps $\bd_{d-1}\D_d^{(I_d)}\setminus\big(\D_{d-1}^{(I_d\setminus\set{r})}\cup\D_{d-1}^{(I_d\setminus\set{s})}\big)$ smoothly onto $\bd_{d-2}\D_{d-1}^{(I_d\setminus\set{s})}$, which together with $\frac{p^{r}}{p^s+p^r}$ being of class $C^\infty$ on $\bd_{d-1}\D_d^{(I_d)}$ (via ${p^s+p^r}>0$) yields the stated regularity; the value of this boundary extension of $\bar\psi$ of course coincides with the one of the corresponding extension of $\psi$.

To prove equation~\eqref{eq_homsol-ext}, \wlg let $I_d=\set{0,1,\dotsc,d}$; summation indices, however, run from~$1$ to~$d$ if nothing differing is stated. To begin with, we have
\begin{align}
L_d^*\left(\psi(\pi^{r,s}(p))\cdot\frac{p^{r}}{p^s+p^r}\right)\notag
=&\left(L_d^* \psi(\pi^{r,s}(p))\right)\frac{p^{r}}{p^s+p^r}\\\notag
&+\sum_{i,j}p^i(\delta^i_j-p^j)\left(\dd{}{p^i}\psi(\pi^{r,s}(p))\right)\left(\dd{}{p^j}\frac{p^{r}}{p^s+p^r}\right)\\
&+\half \psi(\pi^{r,s}(p))\sum_{i,j}p^i(\delta^i_j-p^j)\left(\dd{}{p^i}\dd{}{p^j}\frac{p^{r}}{p^s+p^r}\right)\!.
\end{align}
Next, we will show that the first summand equals $-\kappa\bar{\psi}$, whereas the two other summands vanish on $\D_d^{(I_d)}$.

For the first summand, we use $L^*_{d-1}\psi=-\kappa \psi$ in $\D_{d-1}^{(I_d\setminus\set{s})}$, which holds by assumption. To extend this statement to $\D_d^{(I_d)}$, the interplay of the projection needs to be analysed, for which several cases are distinguished. That is, for $s\neq 0$, $r=0$, the projection $\pi^{0,s}$ yields $\tilde{p}^s=0$ and $\tilde{p}^i=p^i$ for $i\in\set{1,\dotsc,d}\smin\set{s}$, 
hence $\dd{\tilde{p}^m}{p^i}=\delta^m_i(1-\delta^m_s)$, and we have
\begin{align}\notag
L_d^* \psi(\pi^{0,s}(p))
&=\half\sum_{i,j} p^i(\delta^i_j-p^j)\dd{}{p^i}\dd{}{p^j}\psi(\pi^{0,s}(p))\\\notag
&=\half\sum_{m,n}\sum_{i,j} p^i(\delta^i_j-p^j)\delta^m_i(1-\delta^m_s)\delta^n_j(1-\delta^n_s)\dd{}{\tilde{p}^m}\dd{}{\tilde{p}^n}\psi(\tilde{p})\\
&=\half\sum_{m,n\neq s} \tilde{p}^m(\delta^m_n-\tilde{p}^n)\dd{}{\tilde{p}^m}\dd{}{\tilde{p}^n}\psi(\tilde{p})
 =L^*_{d-1}\psi(\tilde{p})\equiv -\kappa\psi(\tilde{p}).
\end{align}

If $s=0$, $r\neq 0$ and hence $\D_{d-1}^{(I_d\setminus\set{0})}=\bigset{\hspace*{-1.35pt}(\tilde{p}^1,\dotsc,\tilde{p}^d)\big\vert \tilde{p}^i>0\text{ for $i=1,\dotsc,d$}, \sum_{i=1}^d\tilde{p}^i=1}$, we have $\tilde{p}^i=p^i$ for $i\in\set{1,\dotsc,d}\smin\set{r}$ and $\tilde{p}^r=p^r+p^0$, thus $\dd{\tilde{p}^m}{p^i}=\delta^m_i-\delta^m_r$. We get: 
\begin{align}\notag
L_d^* \psi(\pi^{r,0}(p))
&=\half\sum_{m,n}\sum_{i,j} p^i(\delta^i_j-p^j)(\delta^m_i-\delta^m_r)(\delta^n_j-\delta^n_r)\dd{}{\tilde{p}^m}\dd{}{\tilde{p}^n}\psi(\tilde{p})\\\notag
&=\half\sum_{m,n} p^m(\delta^m_n-{p}^n)\dd{}{\tilde{p}^m}\dd{}{\tilde{p}^n}\psi(\tilde{p})
 -\half\sum_{n}\sum_{i} p^i(\delta^i_n-p^n) \dd{}{\tilde{p}^r}\dd{}{\tilde{p}^n}\psi(\tilde{p})\\\notag
&\quad -\half\sum_{m}\sum_{j} p^m(\delta^m_j-p^m) \dd{}{\tilde{p}^r}\dd{}{\tilde{p}^n}\psi(\tilde{p})
  +\half\sum_{i,j} p^i(\delta^i_j-p^j) \ddsq{}{\tilde{p}^r}\psi(\tilde{p})\\\notag
&=\half\sum_{m,n} p^m(\delta^m_n-p^n)\dd{}{\tilde{p}^m}\dd{}{\tilde{p}^n}\psi(\tilde{p})
 -\half\sum_{n} p^0 p^n \dd{}{\tilde{p}^r}\dd{}{\tilde{p}^n}\psi(\tilde{p})\\
&\quad -\half\sum_{m} p^m p^0 \dd{}{\tilde{p}^m}\dd{}{\tilde{p}^r}\psi(\tilde{p})
  +\half p^0(1-p^0) \ddsq{}{\tilde{p}^r}\psi(\tilde{p}).
\end{align}
When replacing the remaining $p$-coordinates by $\tilde{p}$ (except for $p^0$, which is missing in $\D_{d-1}^{(I_d\setminus\set{0})}$) via $p^i=\tilde{p}^i-p^0\delta^i_r$ for $i=\set{1,\dotsc,d}$, the expression transforms into:
\begin{align}\notag
L_d^* \psi(\pi^{r,0}(p))
&=\half\sum_{m,n\neq r} \tilde{p}^m(\delta^m_n-\tilde{p}^n)\dd{}{\tilde{p}^m}\dd{}{\tilde{p}^n}\psi(\tilde{p})
 +\half\sum_{n\neq r} (-\tilde{p}^r+p^0) \tilde{p}^n \dd{}{\tilde{p}^r}\dd{}{\tilde{p}^n}\psi(\tilde{p})\\\notag
&\quad +\half\sum_{m\neq r} \tilde{p}^m (-\tilde{p}^r+p^0) \dd{}{\tilde{p}^m}\dd{}{\tilde{p}^r}\psi(\tilde{p})
  +\half (\tilde{p}^r-p^0)(1-\tilde{p}^r+p^0)\x\\\notag
&  \qquad\:\ddsq{}{\tilde{p}^r}\psi(\tilde{p})
-\half\sum_{n\neq r} p^0 \tilde{p}^n \dd{}{\tilde{p}^r}\dd{}{\tilde{p}^n}\psi(\tilde{p})
 -\half\sum_{m\neq r} \tilde{p}^m p^0 \dd{}{\tilde{p}^m}\dd{}{\tilde{p}^r}\psi(\tilde{p})\\\notag
&\quad  -p^0(\tilde{p}^r-p^0)\ddsq{}{\tilde{p}^r}\psi(\tilde{p})
+\half p^0(1-p^0)\ddsq{}{\tilde{p}^r}\psi(\tilde{p})\\
&=\half\sum_{m,n} \tilde{p}^m(\delta^m_n-\tilde{p}^n)\dd{}{\tilde{p}^m}\dd{}{\tilde{p}^n}\psi(\tilde{p})
=L^*_{d-1}\psi(\tilde{p})\equiv -\kappa\psi(\tilde{p}).
\end{align}
The next-to-last equality is due to the fact that in $\D_{d-1}^{(I_d\setminus\set{0})}$ one coordinate is obsolete and consequently $\psi$ is formulated in $d-1$ coordinates (which may be chosen freely). It is straightforward to show that, independently of the choice of the omitted coordinate~$r$, we have $L^*_{d-1}=\half\sum_{m,n\neq r} \tilde{p}^m(\delta^m_n-\tilde{p}^n)\dd{}{\tilde{p}^m}\dd{}{\tilde{p}^n}$ on $\D_{d-1}^{(I_d\setminus\set{0})}$.

Lastly, if $s\neq 0$, $r\neq 0$, the projection $\pi^{r,s}$ yields $\tilde{p}^s=0$, $\tilde{p}^r=p^s+p^r$ and $\tilde{p}^i=p^i$ for the remaining indices, hence $\dd{\tilde{p}^m}{p^i}=\delta^m_i(1-\delta^m_s)+\delta^m_r\delta^i_s$. Then we have:
\begin{align}\notag
L_d^* \psi(\pi^{r,s}(p))\hspace*{-1pt}
&=\half\sum_{i,j} p^i(\delta^i_j-p^j)\dd{}{p^i}\dd{}{p^j}\psi(\pi^{r,s}(p))\\\notag
&=\half\sumstack{m,n,\\
i,j} p^i(\delta^i_j-p^j)(\delta^m_i(1-\delta^m_s)\hspace*{-0.8pt}+\delta^m_r\delta^i_s)(\delta^n_j(1-\delta^n_s)\hspace*{-0.8pt}+\delta^n_r\delta^j_s)\dd{}{\tilde{p}^m}\dd{}{\tilde{p}^n}\psi(\tilde{p})\\\notag
&=\half\sum_{m,n\neq s} p^m(\delta^m_n-p^n)\dd{}{\tilde{p}^m}\dd{}{\tilde{p}^n}\psi(\tilde{p})
 -\half\sum_{n\neq s} p^s p^n \dd{}{\tilde{p}^r}\dd{}{\tilde{p}^n}\psi(\tilde{p})\\
&\quad -\half\sum_{m\neq s} p^m p^s \dd{}{\tilde{p}^m}\dd{}{\tilde{p}^r}\psi(\tilde{p})
  +\half p^s(1-p^s) \ddsq{}{\tilde{p}^r}\psi(\tilde{p}).
\end{align}
Replacing the ${p}$-coordinates works as shown in the preceding case, and thereupon we obtain
\begin{align}
L_d^* \psi(\pi^{r,s}(p))=\half\sum_{m,n\neq s} \tilde{p}^m(\delta^m_n-\tilde{p}^n)\dd{}{\tilde{p}^m}\dd{}{\tilde{p}^n}\psi(\tilde{p})
=L^*_{d-1}\psi(\tilde{p})\equiv -\kappa\psi(\tilde{p}),
\end{align}
thus in total 
\begin{align}
L_d^* \psi(\pi^{r,s}(p))
=L^*_{d-1}\psi(\tilde{p})\equiv -\kappa\psi(\tilde{p})=-\kappa \psi (\pi^{r,s}(p))
\end{align}
for arbitrary $r, s$, which is the desired equality result for the first summand.

To show that the two remaining summands vanish, an analogous case-by-case analysis is necessary. If $s=0$, $r\neq 0$, we have $\frac{p^r}{p^0+p^r}=\frac{p^r}{1-\sum_{l\neq r}p^l}$. Due to (remember $\dd{\tilde{p}^m}{p^i}=\delta^m_i-\delta^m_r$)
\begin{align}
\dd{}{p^r}\psi(\pi^{r,0}(p))=\sum_m \dd{\tilde{p}^m}{p^r}\dd{}{\tilde{p}^m}\psi(\tilde{p})=0,
\end{align}
the second summand equalling
\begin{align}
\sum_{i\neq r}p^i\left(\dd{}{p^i}\psi(\pi^{r,0}(p))\right)
&\underbrace{\sum_{j}(\delta^i_j-p^j)\left(\dd{}{p^j}\frac{p^{r}}{1-\sum_{l\neq r}p^{l}}\right)}\\\notag
&{=\big(1-\sum_{j\neq r} p^j)\frac{p^r}{\big(1-\sum_{l\neq r}p^{l}\big)^2}-p^{r}\frac{1}{1-\sum_{l\neq r}p^{l}}=0}
\end{align}
along with the third summand equalling
\begin{align}\notag
&\half \psi(\pi^{r,0}(p))\sum_{i\neq r}\Bigg(
\sum_{j\neq r}p^i(\delta^i_j-p^j)\left(\dd{}{p^i}\dd{}{p^j}\frac{p^{{r}}}{1-\sum_{l\neq r}p^{l}}\right)\\\notag
&\hspace*{167pt}-2p^ip^{{r}}\left(\dd{}{p^i}\dd{}{p^{r}}\frac{p^{{r}}}{1-\sum_{l\neq r}p^{l}}\right)\Bigg)\\
&=\half \psi(\pi^{r,0}(p))\sum_{i\neq r}\Bigg(
p^i(1-\sum_{j\neq r} p^j)\frac{2p^{r}}{\big(1-\sum_{l\neq r}p^{l}\big)^3}
-2p^ip^{{r}}\frac{1}{\big(1-\sum_{l\neq r}p^{l}\big)^2}
\Bigg)=0
\end{align}
vanish. 

Similarly, if $s\neq 0$, $r=0$, thus $\frac{p^0}{p^s+p^0}=\frac{1-\sum_{l} p^l}{1-\sum_{l\neq s}p^l}$ and again (with $\dd{\tilde{p}^m}{p^i}=\delta^m_i(1-\delta^m_s)$)
\begin{align}
\dd{}{p^s}\psi(\pi^{0,s}(p))=\sum_m \dd{\tilde{p}^m}{p^s}\dd{}{\tilde{p}^m}\psi(\tilde{p})=0,
\end{align}
the second summand equalling
\begin{align}\notag
\sum_{i\neq s}p^i\left(\dd{}{p^i}\psi(\pi^{0,s}(p))\right)
&\underbrace{\sum_{j}(\delta^i_j-p^j)\left(\dd{}{p^j}\frac{1-\sum_{l} p^l}{1-\sum_{l\neq s}p^{l}}\right)}\\
&{=\big(1-\sum_{j} p^j)\frac{-1}{1-\sum_{l\neq s}p^{l}}+\big(1-\sum_{j\neq s} p^{j}\big)\frac{1-\sum_{l}p^{l}}{\big(1-\sum_{l\neq s}p^{l}\big)^2}=0}
\end{align}
vanishes, and the third summand via
\begin{align}\notag
&\sum_{i,j}p^i(\delta^i_j-p^j)\dd{}{p^i}\dd{}{p^j}\frac{1-\sum_{l} p^l}{1-\sum_{l\neq s}p^{l}}\\\notag
&=\sum_{i,j}p^i(\delta^i_j-p^j)\Bigg(\frac{(\delta^i_s-1)+(\delta^j_s-1)}{\big(1-\sum_{l\neq s}p^{l}\big)^2}
+2(1-\delta^i_s)(1-\delta^j_s)\frac{1-\sum_{l} p^l}{\big(1-\sum_{l\neq s}p^{l}\big)^3}\Bigg)\\
&=-2\frac{\big(\sum_{i\neq s} p^i\big)\big(1-\sum_{j\neq s}p^{j}\big)}{\big(1-\sum_{l\neq s}p^{l}\big)^2}
 + 2 \Big(\sum_{i\neq s} p^i\Big)\Big(1-\sum_{j\neq s}p^{j}\Big) \frac{1-\sum_{l} p^l}{\big(1-\sum_{l\neq s}p^{l}\big)^3}
=0
\end{align}
also does.

Ultimately, if $s\neq 0$, $r\neq 0$, we have
\begin{align}
p^j \dd{}{p^j}\frac{p^{r}}{p^s+p^r}=\frac{p^sp^r}{(p^s+p^r)^2}(\delta^j_r-\delta^j_s).
\end{align} 
Using this property for the second summand, we obtain 
\begin{align}\notag
&\sum_{i,j}p^i(\delta^i_j-p^j)\left(\dd{}{p^i}\psi(\pi^{r,s}(p))\right)\left(\dd{}{p^j}\frac{p^{r}}{p^s+p^r}\right)\\\notag
=&\sum_{i}\left(\dd{}{p^i}\psi(\pi^{r,s}(p))\right) p^i\bigg(\sum_j\delta^i_j\left(\dd{}{p^j}\frac{p^{r}}{p^s+p^r}\right)-\sum_j p^j\left(\dd{}{p^j}\frac{p^{r}}{p^s+p^r}\right)\bigg)\\
=&\sum_{i}\dd{}{p^i}\psi(\pi^{r,s}(p))\frac{p^sp^r}{(p^s+p^r)^2}(\delta^i_r-\delta^i_s)=0.
\end{align}
The last equality is due to the fact that the sum over~$i$ in the last line vanishes in conjunction with the symmetry of $\pi$ in the coordinates $p^s$ and $p^r$, \ie we have $\dd{\tilde{p}^m}{p^i}=\delta^m_i(1-\delta^m_s)+\delta^m_r\delta^i_s$ and consequently
\begin{align}
\dd{}{p^s}\psi(\pi^{r,s}(p))=\dd{}{\tilde{p}^r}\psi(\tilde{p})=\dd{}{p^r}\psi(\pi^{r,s}(p)).
\end{align}
For the third summand, we use
\begin{align}
\dd{}{p^i}\dd{}{p^j}\frac{p^{r}}{p^s+p^r}=2\frac{\delta^i_j(\delta^i_s p^r-\delta^i_rp^s)}{(p^s+p^r)^3}+\frac{\delta^i_s\delta^j_r(1-\delta^i_j)(p^r-p^s)}{(p^s+p^r)^3}
\end{align}
and thereon get
\begin{align}
\sum_{i,j}p^i(\delta^i_j-p^j)\dd{}{p^i}\dd{}{p^j}\frac{p^{r}}{p^s+p^r}
&=\frac{2p^s(1-p^s)p^r-2p^r(1-p^r)p^s-2p^sp^r(p^r-p^s)}{(p^s+p^r)^3}=0.
\end{align}

Altogether, we have
\begin{align}
L_d^*\bar{\psi}=L_d^*\left(\psi(\pi^{r,s}(p))\cdot\frac{p^{r}}{p^s+p^r}\right)
=&-\kappa\psi(\pi^{r,s}(p))\frac{p^{r}}{p^s+p^r}
=-\kappa\bar{\psi}
\end{align}
for arbitrary $r,s\in I_d$, thus proving equation~\eqref{eq_homsol-ext}.
\end{proof}

\section{A probabilistic interpretation of the extension scheme}\label{sec_prob_int_n}

We shall now discuss the meaning of the extension constraints~\ref{dfi_ext}. 
A  \txind{target set} on the space of $d-1$ alleles can not only be reached from a constellation of $d-1$ alleles, but also from one of $d$ alleles by allele loss. Therefore,  we need to analyze how the attraction of such a target set also extends to the space of $d$ alleles. A natural assumption for such an extension is that the probability density at the transition from the $d$-allelic domain to the $(d-1)$-allelic domain stays regular, \ie small alterations of the allelic configuration should only affect the probability in a controlled way. This is formulated in condition~(i) and implies the $C_p^\infty$ regularity (\cf equality~\ref{eql_reg_pathwise}) for the corresponding domains. 
Moreover, a boundary condition enters, as for transitions to domains of a different set of $d-1$ alleles, the corresponding probability should also stay regular with the additional requirement that in the limit it vanishes on those other $(d-1)$-allelic domains; this is also part of condition~(i) and correspondingly implies the $C_{p_0}^\infty$ regularity (\cf equality~\ref{eql_reg_ext}). As a possible extension is so far only confined  towards the boundary of the domain, we also wish to link the evolution of the original probability density and its extension by requiring that both are subject to the same type of evolution  in the corresponding domain, \ie are governed by the corresponding \KBE in the relevant formulation, which is condition~(ii).

The extension proposition~\ref{prop_extension} then states that any (proper) solution of the \KBE, which describes the evolving attraction of some target set given via the \txind{final condition} $f$, may be extended to a corresponding solution of the \KBE in the domain of subsequent higher dimension with both conditions above applying. In the context of a \WF model, this loss of the extra allele~$s$ is modelled as if it was in competition with just one other allele~$r$ dependent on the index chosen \textit{(\txind{fibration property}\label{fibration_prop})}. Thus, we say that allele~$s$ is lost over allele~$r$.

However, as may be observed by remark~\ref{rmk_ext_t_0}, this extension actually yields the solution to a somewhat altered problem, namely the attraction generated by the \txind{target set} itself plus an induced (generalised) target set in the bigger domain which are given by $f$ and its corresponding extension $\bar{f}$. If one wishes to return to the original problem,  the attraction of the original target set only located in the $(d-1)$-allelic domain, the induced target set needs to be compensated for by a proper solution in (the interior of) the $d$-allelic domain for a corresponding final condition. 

As may also be seen in proposition~\ref{prop_extension}, for $d\geq2$ the given extension scheme involves a potential ambiguity regarding the choice of the extension target face index~$r$. However, in case of iterations, the boundary condition in definition~\ref{dfi_ext}~(i)
limits this to a unique appropriate value as will be demonstrated in the next section; for a simple extension from a 0-dimen\-sional domain or if the starting distribution smoothly vanishes towards all boundaries of subsequent lower dimension (as with proper solutions), an extension is always in accordance with the boundary condition.

\section{Iterated extensions}

A repeated application of proposition~\ref{prop_extension} yields the existence of iterated extensions (generalising the corresponding result for $n=2$ in~\cite{littler_loss} and the (less explicit) result stated in~\cite{littler-good_ages} without derivation):

\begin{prop}[pathwise extension of solutions]\label{prop_ext_iter}
Let $k,n\in\N$ with $0\leq k<n$, $\set{i_k,i_{k+1},\dotsc,i_n}\subset I_n\ce\set{0,1,\dotsc,n}$ with $i_i\neq i_j$ for $i\neq j$ and $I_k\ce I_n\setminus\set{i_{k+1},\dotsc,i_n}$, and let $u_{I_k}$ \sindex[not]{uik@$u_{I_k}$} be a proper solution of the \KBE \eqref{eq_back_n_ext} in ${\D_k^{(I_k)}}$ for some final condition $ f\in \L^2\big({\D_k^{(I_k)}}\big)$ as  in proposition~\ref{prop_sol_back_n}. 
For $d=k+1,\dotsc,n$ and $I_{d}\ce I_k \cup \set{i_{k+1},\dotsc i_{d}}$, an extension of $\bar{u}_{I_k}^{i_k,\dotsc,i_{d-1}}$ in $\big(\D_{d-1}^{(I_{d-1})}\big)_{-\infty}$ to  $\bar{u}_{I_k}^{i_k,\dotsc,i_{d}}\ce{\big(\bar{u}_{I_k}^{i_k,\dotsc,i_{d-1}}\big)}^{i_{d-1},i_d}$ in $\big(\D_d^{(I_{d})}\big)_{-\infty}$ as by proposition~\ref{prop_extension} is in accordance with the extension constraints~\ref{dfi_ext} if (and for $d\geq k+2$ and $[f]\neq0$ in $L^2\big({\D_k^{(I_k)}}\big)$ also only if) putting $r(d)=i_{d-1}$ for the extension target face index, and we respectively have 
%
\sindex[not]{uikid@$\bar{u}_{I_k}^{i_k,\dotsc,i_{d}}$}\sindex[not]{pikid@$\pi^{i_k,\dotsc,i_d}$}
\begin{align}\label{eq_iter_ext}
\bar{u}_{I_k}^{i_k,\dotsc,i_{d}}(p,t)=
u_{I_k}(\pi^{i_k,\dotsc,i_d}(p),t)\prod_{j=k}^{d-1} \frac{p^{i_j}}{\sum_{l=j}^d p^{i_l}},\quad (p,t)\in\big(\D_d^{(I_d)}\big)_{-\infty}
\end{align}
with $p^0=1-\sum_{i\in I_d\setminus\set{0}}p^i$ and $\pi^{i_k,\dotsc,i_d}(p)=(\tilde{p}^1,\dotsc,\tilde{p}^n)$
such that $\tilde{p}^{i_k}=p^{i_k}+\dotso+p^{i_d}$, $\tilde{p}^{i_{k+1}}=\dotso=\tilde{p}^{i_{d}}=0$ and $\tilde{p}^j=p^j$ for $j\in I_d\smin\set{i_k,\dotsc,i_d}$.


Correspondingly, the resulting assembling of all extensions to a function  $\bar{U}_{I_k}^{i_k,\dotsc,i_{n}}$ in $\Big(\bigcup_{ k \leq d\leq n}\D_d^{(I_d)}\Big)_{-\infty}$
by putting\sindex[not]{Uikid@${U}_{I_k}^{i_k,\dotsc,i_{n}}$}
\begin{multline}\label{eq_path_ext}
\bar{U}_{I_k}^{i_k,\dotsc,i_{n}}(p,t) \ce u_{I_k}(p,t)\ind_{\D_k^{(I_k)}}(p)+\sum_{ k+1\leq d\leq n}\bar{u}_{I_k}^{i_k,\dotsc,i_{d}}(p,t)\ind_{\D_d^{(I_d)}}(p)\\
=u_{I_k}(p,t)\ind_{\D_k^{(I_k)}}(p)+\sum_{ k+1\leq d\leq n} u_{I_k}(\pi^{i_k,\dotsc,i_d}(p),t)\prod_{j=k}^{d-1} \frac{p^{i_j}}{\sum_{l=j}^d p^{i_l}} \ind_{\D_d^{(I_d)}}(p)
\end{multline}
with $p^0=1-\sum_{i\in I_n\setminus\set{0}}p^i$ is in $C_{p_0}^\infty\Big(\bigcup_{ k \leq d\leq n}\D_d^{(I_d)}\Big)$ with respect to the spatial variables for $t<0$ 
as well as in $C^\infty((-\infty,0))$ with respect to $t$, and we have
\begin{align}
\begin{cases}\label{eq_iter_ext_sol}
L^*\bar{U}_{I_k}^{i_k,\dotsc,i_{n}}=-\dd{}{t}\bar{U}_{I_k}^{i_k,\dotsc,i_{n}}
&\text{in $\Big(\bigcup_{ k \leq d\leq n}\D_d^{(I_d)}\Big)_{-\infty}$}\\
\bar{U}_{I_k}^{i_k,\dotsc,i_{n}}(\fdt,0)=\bar{F}_{I_k}^{i_k,\dotsc,i_{n}}
&\text{in $\bigcup_{ k \leq d\leq n}\D_d^{(I_d)}$}
\end{cases}
\end{align}
with $\bar{F}_{I_k}^{i_k,\dotsc,i_{n}}\in \L^2\Big(\bigcup_{ k \leq d\leq n}\D_d^{(I_d)}\Big)$ being an analogous extension of the final condition $f=f_{I_k}$ in $\D_k^{(I_k)}$  as by remark~\ref{rmk_ext_t_0}; in particular, we have $\bar{U}_{I_k}^{i_k,\dotsc,i_{n}}\big|_{{\D_k^{(I_k)}}}(\fdt,0)=f$ in~${{\D_k^{(I_k)}}}$.
\end{prop}

\begin{cor}\label{cor_ext_iter_stat}
For $n\in\Np$, $k=0$ and $u_{\set{i_0}}\equiv 1$ in $\D_0^{(\set{i_0})}\subset\bd_0\D_n$, equation~\eqref{eq_iter_ext} \resp equation~\eqref{eq_path_ext} restricted to $\D_n$ and with the $t$-coordinate suppressed coincides with Littler's formula in $\D_n$ (\cf~\cite{littler-good_ages}):
\begin{align}
\bar{U}_{\set{i_0}}^{i_0,i_1\dotsc,i_{n}}\big\vert_{\D_n}(p)\equiv
\bar{u}_{\set{i_0}}^{i_0,i_1\dotsc,i_{n}}(p)=p^{i_0}\cdot\frac{p^{i_1}}{1-p^{i_0}}\cdot\dotso\cdot\frac{p^{i_{n-1}}}{1-\sum_{l=0}^{n-2}p^{i_l}}.
\end{align}
\end{cor}

\begin{proof}[Proof of proposition~\ref{prop_ext_iter}]
The result is basically an application of proposition~\ref{prop_extension}, which yields the regularity and the solution property (\cf equation~\eqref{eq_iter_ext_sol}) in every $\D_d^{(I_d)}$. It only remains to show inductively that the boundary behaviour in each extension step respects the extension constraints~\ref{dfi_ext} as well as the formula~\eqref{eq_iter_ext}. 

Clearly, a proper solution $u_{I_k}$ of the \KBE in $\big(\D_k^{(I_k)}\big)_{-\infty}$ as in proposition~\ref{prop_sol_back_n} satisfies equation~\eqref{eq_iter_ext} and is of class $C^\infty_0\big(\D_k^{(I_k)}\big)$ \wrt the spatial variables for $t<0$ (which in particular implies that it is smoothly extendable to $\bd_{k-1}\D_k^{(I_k)}$). 
Extending $u_{I_k}$ to $\big(\D_{k+1}^{(I_{{k+1}})}\big)_{-\infty}$ via proposition~\ref{prop_extension} with $s({k+1})=i_{{k+1}}$ and $r({k+1})=i_{k}$ yields a function $\bar{u}_{I_k}^{i_k,i_{k+1}}$ of type~\eqref{eq_iter_ext}, which for $t<0$ smoothly extends to all boundary faces $\bd_{k}\D_{{k+1}}^{(I_{{k+1}})}$ and vanishes there except for  $\D_{k}^{(I_{k})}$ (where it coincides with $u_{I_k}$) by the assumed boundary behaviour of $u_{I_k}$. We may thus assume that for $k<d-1<n$ an assembled extension $\bar{U}_{I_k}^{i_k,\dotsc,i_{d-1}}$ (corresponding to equation~\eqref{eq_path_ext}) in $C_{p_0}^\infty\big(\bigcup_{ k \leq m\leq d-1}\D_m^{(I_m)}\big)$ with respect to the spatial coordinates exists whose top-dimen\-sional component $\bar{U}_{I_k}^{i_k,\dotsc,i_{d-1}}\big|_{\big(\D_{d-1}^{(I_{d-1})}\big)_{-\infty}}\ec \bar{u}_{I_k}^{i_k,\dotsc,i_{d-1}}$ satisfies equation~\eqref{eq_iter_ext}.

We may then perform an extension of $\bar{u}_{I_k}^{i_k,\dotsc,i_{d-1}}$ in ${\big(\D_{d-1}^{(I_{d-1})}\big)_{-\infty}}$ to $\bar{u}_{I_k}^{i_k,\dotsc,i_{d}}$ in $\big(\D_d^{(I_{d})}\big)_{-\infty}$ via proposition~\ref{prop_extension} with $s(d)=i_{d}$ and $r(d)=i_{d-1}$. By the assumed boundary behaviour of $\bar{u}_{I_k}^{i_k,\dotsc,i_{d-1}}$ (\ie $\bar{U}_{I_k}^{i_k,\dotsc,i_{d-1}}$ being of class $C_{p_0}^\infty$), $\bar{u}_{I_k}^{i_k,\dotsc,i_{d}}$ smoothly extends to all boundary faces $\bd_{d-1}\D_{d}^{(I_{d})}\smin\D_{d-1}^{(I_{d}\setminus\set{i_{d-1}})}$ and vanishes there except for $\D_{d-1}^{(I_{d-1})}$ (where it coincides with $\bar{u}_{I_k}^{i_k,\dotsc,i_{d-1}}$) for $t<0$. By putting $r(d)=i_{d-1}$, this particularly also holds for $\D_{d-1}^{(I_{d}\setminus\set{i_{d-1}})}$, which in turn would otherwise be violated if $f\neq0$ almost everywhere as may be seen from the proof of proposition~\ref{prop_extension}. Then, the boundary behaviour respects the extension constraints~\ref{dfi_ext}, and we correspondingly have $\bar{U}_{I_k}^{i_k,\dotsc,i_{d}}\ce \bar{U}_{I_k}^{i_k,\dotsc,i_{d-1}}+\bar{u}_{I_k}^{i_k,\dotsc,i_{d}}\ind_{\D_d^{(I_d)}}\in C_{p_0}^\infty\big(\bigcup_{k \leq m\leq d}\D_m^{(I_m)}\big)$ \wrt the spatial variables for $t<0$.

To show equation~\eqref{eq_iter_ext}, we obtain for $\bar{u}_{I_k}^{i_k,\dotsc,i_{d}}$ by equation~\eqref{eq_extension} when plugging in the formula~\eqref{eq_iter_ext} for $\bar{u}_{I_k}^{i_k,\dotsc,i_{d-1}}$ 
\begin{align}\notag
\bar{u}_{I_k}^{i_k,\dotsc,i_{d}}(p,t)&=\bar{u}_{I_k}^{i_k,\dotsc,i_{d-1}}(\pi^{i_{d-1},i_{d}}(p),t)\frac{p^{i_{d-1}}}{p^{i_{d-1}}+p^{i_{d}}}\\\notag
&= u_{I_k}(\pi^{i_k,\dotsc,i_{d-1}}(\pi^{i_{d-1},i_{d}}(p)),t)\prod_{j=k}^{d-2} \frac{(\pi^{i_{d-1},i_{d}}(p))^{i_j}}{\sum_{l=j}^{d-1} (\pi^{i_{d-1},i_{d}}(p))^{i_l}} \frac{p^{i_{d-1}}}{p^{i_{d-1}}+p^{i_{d}}}\\
&=u_{I_k}(\pi^{i_k,\dotsc,i_{d}}(p),t)\prod_{j=k}^{d-1} \frac{p^{i_j}}{\sum_{l=j}^{d} p^{i_l}}
\qquad\text{in $\big(\D_d^{(I_{d})}\big)_{-\infty}$}
\end{align}
as $(\pi^{i_{d-1},i_{d}}(p))^{i_j}={p}^{i_j}$ for $i_j=i_k,\dotsc,i_{d-2}$ and $(\pi^{i_{d-1},i_{d}}(p))^{i_{d-1}}={p}^{i_{d-1}}+{p}^{i_{d}}$. 
If some index $i_j$ equals zero (\wlg $i_0=0$) corresponding to $(\pi^{i_{d-1},i_{d}}(p))^0$, this expression gets replaced by $p^0\in\D_d^{(I_{d})}$ as we have $(\pi^{i_{d-1},i_{d}}(p))^0=1-\sum_{j=1}^{d-1}{(\pi^{i_{d-1},i_{d}}(p))^{i_j}}=1-\sum_{j=1}^{d}p^{i_j}\equiv p^0$.
Furthermore, $\pi^{i_k,\dotsc,i_{d-1}}(\pi^{i_{d-1},i_{d}}(p))=\pi^{i_k,\dotsc,i_{d}}(p)$ directly follows from the definitions, thus proving equation~\eqref{eq_iter_ext} for $\bar{u}_{I_k}^{i_k,\dotsc,i_{d}}$ in $\Big(\bigcup_{k\leq m \leq d}\D_m^{(I_m)}\Big)_{-\infty}$.
\end{proof}

\begin{rmk}
Geometrically, the choice of the extension target face indices $s(d)=i_{d}$ and $r(d)=i_{d-1}$ signifies that the extension source face $\D_{d-1}^{(\set{i_0,\dotsc,i_{d-2},i_{d-1}})}$ and the target face $\D_{d-1}^{(\set{i_0,\dotsc,i_{d-2},i_{d}})}$ are adjacent faces to the highest degree, as they share $d-1$ vertices (for $d\geq 2$). Furthermore, their intersection $\D_{d-2}^{(\set{i_0,\dotsc,i_{d-2}})}$ is the extension source face of the previous step.
\end{rmk}

Sticking to the preceding \txind{probabilistic interpretation}, $\bar{u}_{I_k}^{i_k,i_{k+1},\dotsc,i_{n}}$ depicts the iterated `attraction' of an (analogously extended) target set in $\D_k^{(I_k)}$ along a corresponding \itind{extension path} \label{pag_ext_path} specified by $i_k,\dotsc,i_n$ \resp the corresponding index sets $I_k\subset\dotso\subset I_n$. Thus, $\bar{u}_{I_k}^{i_k,i_{k+1},\dotsc,i_{n}}$ gives the total probability for all paths in $\cl{\D}_n$ starting in $\D_n^{(I_n)}$, passing through the (sub)simplices
\begin{align}\label{eq_path_downwards}
\D_{n-1}^{(I_{n-1})}\too\D_{n-2}^{(I_{n-1})}\too\dotso\too\D_{k+1}^{(I_{k+1})}\too\D_k^{(I_k)}
\end{align}
and reaching the eventual \txind{target set}, which, in the setting of the \WF model, corresponds to eventually losing $n-k$ of originally~$n$ alleles in such a manner that from dimension $n-1$ down to~$0$ exactly the allele sets
\begin{align}\label{eq_path_downwards_alleles}
I_n\too I_{n-1}\too\dotso\too I_{k+1}\too I_k
\end{align}
are present until reaching the eventual target set. 

As depicted, these pathwise extensions are a consequence of the boundary condition of the extension constraints~\ref{dfi_ext}:
On the one hand, there is only one allele which is lost at a certain time; on the other hand, as this loss is modelled as if it was in competition with just one other allele, the corresponding allele always is the one which is lost next. Thus allele $i_d$ is lost over $i_{d-1}$; merely in the last step, \ie the loss of allele $i_{k+1}$, the index $i_k$ determines which of the alleles in $I_k$ is the one $i_{k+1}$ is lost over. 
Other extensions which may likewise be constructed by the extension lemma~\ref{lem_extension} will not be considered here.

However, the corresponding extensions in proposition~\ref{prop_ext_iter} are not satisfactory to the extent that they lack a global (pathwise) regularity property on the entire $\cl{\D}_n$, \ie are not in $C_p^\infty$ \wrt the spatial variables, as this applies only along the corresponding \txind{extension path}. Outside this path, generally no continuous (or even smooth) extensions exist. This is caused by the incompatibilities involved by this construction (\cf also section~\ref{sec_prob_int_n}): For example on $\D_{k+1}^{(\tilde{I}_{k+1})}$ with $\tilde{I}_{k+1}\ce I_k\cup \set{{\tilde{\imath}_k}}$ and ${\tilde{\imath}_k} \in I_n\smin I_{k+1}$, a positive hit probability for the target set in $\D_k^{(I_k)}$ by a direct loss of allele $\tilde{\imath}_k$ would exist, yet the considered solution necessarily vanishes on $\D_{k+1}^{(\tilde{I}_{k+1})}$ as this is a boundary face of $\D_{k+2}^{(I_{k+2})}$ outside the specified path. 

This defect is overcome by mounting these extensions into a global solution covering all possible extensions paths, each one of them corresponding to a certain ordering of the indices in $I_n\setminus I_k$. As in the first extension step, the extension target face is not defined for a given extension path and a non-empty target set by the extension boundary condition~(i) in definition~\ref{dfi_ext} (except for $k=0$; \cf proposition~\ref{prop_ext_iter}), correspondingly all indices in $I_k$ may serve as target face index. This is taken into account by additionally summing over all possible first stage extensions and normalising, yielding in total:

\begin{prop}[global extension of solutions]\label{prop_ext_iter_glob}
Let $k,n\in\N$ with $0\leq k<n$,
$I_k\subset I_n\ce\set{0,1,\dotsc,n}$ with $\abs{I_k}=k+1$, 
and let $u_{I_k}$ 
be a proper solution of the \KBE \eqref{eq_back_n_ext} in ${\D_k^{(I_k)}}$ for some final condition $f\in\L^2\big({\D_k^{(I_k)}}\big)$ as  in proposition~\ref{prop_sol_back_n}. 
Then an assembling of all pathwise extensions of $u_{I_k}$ 
as by proposition~\ref{prop_ext_iter} into a function\sindex[not]{Uik@$\bar{U}_{I_k}$} $\bar{U}_{I_k}\in\big(\cl{\D}_n\big)_{-\infty}$ by putting\footnote{The last sum actually only comprises a single summand; this notation is used to illustrate the choice of the index $i_d$, however.}
\begin{multline}\label{eq_glob_ext}
\bar{U}_{I_k}(p,t)\ce
u_{I_k}(p,t)\ind_{\D_k^{(I_k)}}(p)\\
+\frac{1}{\abs{I_k}}\,\sum_{i_k\in I_k}\,\sum_{ k+1\leq  d\leq n}\,\sum_{ i_{k+1}\in I_n\smin I_k}\dotsc
\sumstack{i_d\in I_n\smin (I_k\cup \\\set{i_{k+1},\dotsc,i_{d-1}})}\bar{u}_{I_k}^{i_k,\dotsc,i_{d}}(p,t)
\ind_{\D_d^{(I_k\cup \set{i_{k+1},\dotsc,i_{d}})}}(p)
\end{multline}
for $(p,t)\in\big(\bigcup_{ I_k \subset I_d\subset I_n}\D_d^{(I_d)}\big)_{-\infty}$ and  $\bar{U}_{I_k}(p,t)\ce 0$ in the remainder of $\big(\cl{\D}_n\big)_{-\infty}$ is in $C_p^\infty\big(\cl{\D}_n\big)$ with respect to the spatial variables for $t<0$ as well as in $C^\infty((-\infty,0))$ with respect to $t$.
Furthermore,  $\bar{U}_{I_k}$ is a solution of the corresponding \KBE in $\big(\cl{\D}_n\big)_{-\infty}$ and for $t=0$ matches an analogously assembled extension $\bar{F}_{I_k}$ of $f=f_{I_k}$ in ${\D_k^{(I_k)}}$ as final condition in $\cl{\D}_n$ (\cf remark~\ref{rmk_ext_t_0}).
\end{prop}

\begin{proof}
The asserted global regularity directly follows from properties of the applied extension scheme as stated in lemma~\ref{lem_extension} and proposition~\ref{prop_ext_iter} and the construction of $\bar{U}_{I_k}$, which is such that potential discontinuities are ruled out by assembling all extensions along arbitrary paths. The solution property and the compliance with the analogously constructed final condition likewise straightforwardly extend from proposition~\ref{prop_ext_iter}.
\end{proof}

Shifting again to the  \txind{probabilistic interpretation}, $\bar{U}_{I_k}$ now depicts the full iterated `attraction' of some eventual \txind{target set} in $\D_k^{(I_k)}$ and its (successively) induced target sets in $\D_d^{(I_d)}\subset\cl{\D}_n$ with $I_d\supset I_k$, which may now be reached along arbitrary paths. Thus, $\bar{U}_{I_k}$ gives the total probability for all paths from $\D_n^{(I_n)}$ to eventually $\D_k^{(I_k)}$  --  with no assumptions on possible interstages made.
In the setting of the \WF model, this corresponds to eventually losing $n-k$ of previously~$n$ alleles irrespective of any order of loss. 

Since $\bar{U}_{I_k}$ represents the most general extension of a given solution $u_{I_k}$ in $\D_k^{(I_k)}$ to $\cl{\D}_n$, the general solution scheme for solutions of the extended \KBE~\eqref{eq_back_n_ext} may now be developed.

\section{Construction of general solutions via the extension scheme}\label{sec_back_gen_sol}

For a given final condition $f=\sum_{d=0}^n f_d\ind_{\bd_d\D_n}\in\LII\big(\bigcup_{d=0}^n\bd_d\D_n\big)$, the following extension scheme  allows us to construct a solution of the extended \KBE \eqref{eq_back_n_ext} which captures the full dynamics of the process on the entire $\big(\cl{\D}_n\big)_{-\infty}$. 
The main ingredient for this are the global extensions of a (proper) solution of the \KBE in every instance of the domain as in proposition~\ref{prop_ext_iter_glob}; these globally extended solutions are superposed in a way that eventually the given final condition is met in the entire $\cl{\D}_n$ (\cf also section~\ref{sec_prob_int_n} for a probabilistic interpretation).

Thus, first equation~\eqref{eq_back_n_ext} is solved in each $\big(\D_0^{(\set{i_0})}\big)_{-\infty}\subset(\bd_0{\D}_n)_{-\infty}$ for the final condition $f_0$, and afterwards, these solutions are successively extended to $\big(\cl{\D}_n\big)_{-\infty}$ by means of proposition~\ref{prop_ext_iter_glob}, which analogously generates a successively extended final condition in $\cl{\D}_n$ for $t=0$.  
Subsequently, a (proper) solution in each $\big(\D_1^{(I_1)}\big)_{-\infty}\subset( \bd_1\D_n)_{-\infty}$ for the final condition $f_1$ minus the extension of $f_0$ is determined, which is then successively extended to $\big(\cl{\D}_n\big)_{-\infty}$ (again likewise generating an analogously extended final condition). This procedure is repeated until after finding a (proper) solution in $(\D_n)_{-\infty}$ an extended solution in the entire $\big(\cl{\D}_n\big)_{-\infty}$ is determined. 

A solution of the extended \KBE~\eqref{eq_back_n_ext} restricted to some $\big(\D_0^{(\set{i_0})}\big)_{-\infty}\subset(\bd_0\D_n)_{-\infty}$ is -- of course -- trivial, \ie $u_{\set{i_0}}(p,t)=f_0(p)$ for $(p,t)\in\big(\D_0^{(\set{i_0})}\big)_{-\infty}$, and 
by proposition~\ref{prop_ext_iter_glob} we obtain $\bar{U}_{\set{i_0}}$
as an extension to $\big(\cl{\D}_n\big)_{-\infty}$.
Summing over all $\D_0^{(\set{i_0})}$ yields
\begin{align}
\bar{U}_0\ce \sum_{\set{i_0}\subset I_n} \bar{U}_{\set{i_0}}\qquad\text{in $\big(\cl{\D}_n\big)_{-\infty}$}
\end{align}
with $\bar{U}_0$ in $C_p^\infty\big(\cl{\D}_n\big)$ with respect to the spatial variables as well as in $C^\infty((-\infty,0))$ with respect to $t$  and 
\begin{align}
\begin{cases}
L^*\bar{U}_0=-\dd{}{t}\bar{U}_0
&\text{in $\big(\cl{\D}_n\big)_{-\infty}$}\\
\bar{U}_0
(\fdt,0)=\bar{F}^\prime_0
&\text{in $\cl{\D}_n$}
\end{cases}
\end{align}
with $\bar{F}_0^\prime$ being a corresponding superposed global extension of all $f_0^\prime\equiv f_0$ in $\bd_0\D_n$ as described above for the $u_{\set{i_0}}$ (\cf also remark~\ref{rmk_ext_t_0}), in particular we have $\bar{U}_0|_{{\bd_0\D_n}}(\fdt,0)=f_0$.

For the  next step, proper solutions in $(\bd_1\D_n)_{-\infty}$ are determined and likewise extended to $\big(\cl{\D}_n\big)_{-\infty}$. However, as this extension procedure will be repeated for all $d$-dimen\-sional instances of $(\D_n)_{-\infty}$ for $d=1,\dotsc,n$, we directly assume that suitable solutions in $\big(\bigcup_{m=0}^{d-1}\bd_m\D_n\big)_{-\infty}$ already have successively been determined and extended to $\big(\cl{\D}_n\big)_{-\infty}$. Thus $\sum_{m=0}^{d-1} \bar{U}_m$ solves the extended \KBE \eqref{eq_back_n_ext} in $\big(\cl{\D}_n\big)_{-\infty}$ and matches the final condition $f$ for $t=0$ in $\bigcup_{m=0}^{d-1}\bd_m\D_n$ (still, with $\bar{U}_{0}(\fdt,0),\dotsc,\bar{U}_{d-1}(\fdt,0)$ in $\cl{\D}_n$ respectively matching a corresponding superposed global extension $\bar{F}^\prime_m$  of the final condition $f^\prime_m$ in $\bd_m\D_n$ modified as below). Then, a proper solution $u_{I_d}$ by proposition~\ref{prop_sol_back_n} in each $\big(\D_d^{(I_d)}\big)_{-\infty}\subset({\bd_d\D_n})_{-\infty}$, $I_d\subset I_n$ is determined which matches the modified final condition
\begin{align}
f_d^{\prime}\ce f_d-\sum_{m=0}^{d-1}{\bar{F}}_{m}^\prime\vert_{{\bd_d\D_n}}\quad\text{in $\bd_d\D_n$},
\end{align}
correspondingly restricted to the relevant $\D_d^{(I_d)}$. For each $I_d$, 
the solution $u_{I_d}$ is then extended to $\big(\cl{\D}_n\big)_{-\infty}$ via proposition~\ref{prop_ext_iter_glob} each leading to a function $\bar{U}_{I_d}$.
Clearly, these extensions do not interfere with the solutions on lower dimensional entities by definition. 

Summing over the extensions of all $u_{I_d}$, $I_d\subset I_n$, we obtain\sindex[not]{Ud@$\bar{U}_d$}
\begin{align}
{\bar{U}}_d\ce\sum_{I_d\subset I_n} \bar{U}_{I_d}\quad\text{in $\big(\cl{\D}_n\big)_{-\infty}$}
\end{align}
as the global extension of all (proper) solutions in $(\bd_d\D_n)_{-\infty}$. By proposition~\ref{prop_ext_iter_glob} and the linearity of the differential equation, $\bar{U}_d$ is in $C_p^\infty\big(\cl{\D}_n\big)$ \wrt the spatial variables as well as in $C^\infty((-\infty,0)$ with respect to $t$ and solves the extended \KBE and for $t=0$ matches a corresponding superposed global extension $\bar{F}^\prime_d$ of the final condition $f^\prime_d$ in ${\bd_d\D_n}$, thus in particular ${\bar{U}}_d(\fdt,0)|_{\bd_d\D_n}=f^\prime_d$. 
Consequently, the sum of all up to now extended solutions also is in $C_p^\infty\big(\cl{\D}_n\big)$ \wrt the spatial variables  as well as in $C^\infty((-\infty,0)$ with respect to $t$ and satifies
\begin{align}
\begin{cases}
L^*\Big(\sum_{m=0}^d \bar{U}_m\Big)=-\dd{}{t}\Big(\sum_{m=0}^d \bar{U}_m\Big)
&\text{in $\big(\cl{\D}_n\big)_{-\infty}$}\\
\Big(\sum_{m=0}^d \bar{U}_m\Big)|_{\bigcup_{m=0}^d\bd_m\D_n}(\fdt,0)=f\vert_{\bigcup_{m=0}^d\bd_m\D_n}
&\text{in $\bigcup_{m=0}^d\bd_m\D_n$}.
\end{cases}
\end{align}

Repeating the preceding step successively one eventually arrives at $\sum_{m=0}^{n-1} \bar{U}_m$. For the remaining $(\D_n)_{-\infty}$, at last a (proper) solution $u_{I_n}\ec \bar{U}_n$ by proposition~\ref{prop_sol_back_n} is determined matching the modified final condition
\begin{align}
f^{\prime}_n\ce f_n-\sum_{m=0}^{n-1}{\bar{F}}_{m}^\prime\vert_{\D_n}\quad\text{in $\D_n$}.
\end{align}
Then the sum of all globally extended (proper) solutions in all instances of the domain\sindex[not]{U@$\bar{U}$}
\begin{align}
\bar{U}\ce \sum_{j=0}^{n}{\bar{U}}_{j}
\end{align}
is in $C_p^\infty\big(\cl{\D}_n\big)$  \wrt the spatial variables  as well as in $C^\infty((-\infty,0))$ with respect to~$t$ and satifies
\begin{align}
\begin{cases}
L^*\bar{U}=-\dd{}{t}\bar{U}
&\text{in $\big(\cl{\D}_n\big)_{-\infty}$}\\
\bar{U}(\fdt,0)=f
&\text{in $\cl{\D}_n$}, 
\end{cases}
\end{align}
thus is a solution of the extended \KBE\eqref{eq_back_n_ext}.

Altogether, we have the following existence result:
\begin{thm}\label{thm_sol_back_n_ext}
For a given final condition $f\in\L^2\big(\bigcup_{d=0}^n\bd_d\D_n\big)$,  the extended  \KBE\eqref{eq_back_n} corresponding to the $n$-dimen\-sional \WF model  in diffusion approximation  always allows a solution 
$\bar{U}\co{\big(\overline{\Delta}_{n}\big)}_{-\infty}\map\R$ with $\bar{U}(\,\cdot\,,t)\in C_p^\infty\big(\cl{\D}_n\big)$ for each fixed $t\in(-\infty,0)$  
and $\bar{U}(p,\,\cdot\,)\in C^\infty((-\infty,0))$ for each fixed $p\in\overline{\Delta}_{n}$.
\end{thm}


In a following paper, we will be able to show that for $f\in\L^2\big(\bd_0\D_n\big)$ --  and under some additional regularity assumptions -- the solution obtained, \ie $\bar{U}_0$, also is the unique solution given the described extension scheme.

\section{The stationary \KBE}\label{sec_long-term}

Asking for the long-term behaviour of the process, \ie which alleles are eventually lost and in which order, leads us to a stationary version of the \KBE; solutions thereof have already appeared implicitly in the preceding section as extensions of solutions in $\bd_0\D_n$ since the corresponding operator $L^*_0$ only possesses the eigenvalue 0. 

Even with the extended setting presented in section~\ref{sec_hier_sol_bkw} available, we at first consider some interior simplex $\D_n$, (\resp the corresponding restriction of an extended solution). Then, for a solution in $\D_n$, we may argue again that all eigenmodes of the solution corresponding to a positive eigenvalue vanish for $t\to -\infty$, while those corresponding to the eigenvalue zero are preserved. Thus, it may be shown that a solution of the \KBE~\eqref{eq_back_n} in ${\Delta}_n$ converges uniformly to a solution of the corresponding \textit{homogeneous} or \textit{stationary \KBE} 
\begin{equation}\label{eq_back_n_stat}
\begin{cases}
L^* u(p)=0 		&\text{in ${\Delta}_n$}\\
u(p) = f(p)	&\text{in $\bd\D_n$}\\
\end{cases}
\end{equation}
for $u\in  C^2({\Delta}_n)$ and with boundary condition $f$ (which needs to be attained smoothly in a certain sense).


At first sight, this appears as a boundary value problem (for some suitably chosen boundary function $f$, assuring the uniqueness of a solution). However, as may be expected from the previous considerations, the role of the boundary here is different from usual boundary value problems and again requires some extra care: On the one hand, a proper solution in $\D_n$ always converges to the trivial stationary solution (\ie constantly equalling 0), which is linked to the fact that their (continuous) extension to the boundary also vanishes at all negative times. On the other hand, any solution which extends to $\bd\D_n$ is already strongly constrained by the degeneracy behaviour of the differential operator if suitable regularity assumptions on the solution in $\cl{\D}_n$ (\cf also equality~\eqref{eql_reg_pathwise}) apply:

\begin{lem}[stem lemma]\label{lem_stem}
For a solution $u\in C^\infty(\D_n)$ of equation~\eqref{eq_back_n_stat} with extension $U\in C_p^\infty \big(\cl{\D}_n\big)$,  
we have 
\begin{align}\label{eq_L_closure}
L^*U =0\qquad\text{in $\overline{\Delta}_n$.}
\end{align}
\end{lem}

\begin{proof}
The statement is proven iteratively: Assuming that we have $L^*_k U=0$ for all $\D_k^{(I_{k})}\subset\bd_k\D_n$, we show that this property extends to each $\D_{k-1}^{(I_{k-1})}\subset\bd_{k-1}\D_k^{(I_{k})}$ 
for every $\D_k^{(I_{k})}$, hence we obtain $L_{k-1}^*U=0$ on $\bd_{k-1}\D_n$. A repeated application then yields equation~\eqref{eq_L_closure}.

\Wlg let $\D_k^{(I_{k})}$ and $\D_{k-1}^{(I_{k-1})}\subset\bd_{k-1}\D_k^{(I_{k})}$ with $I_k\smin I_{k-1}=\set{i_k}$. Then for the operator $L^*_{k}$ in $\D_k^{(I_{k})}$, we have
\begin{align}
L^*_{k}= L^*_{k-1}+ p^{i_k}\bigg(\sum_{{i_j}\in I_k\smin\set{0}} (\delta^{i_j}_{i_k}-p^{i_j})\dd{}{p^{i_j}}\dd{}{p^{i_k}}\bigg)
\end{align}
with $L^*_{k-1}$ being the restriction of $L^*_{k}$ to $\D_{k-1}^{(I_{k-1})}$.

Now, choosing some $p\in\D_{k-1}^{(I_{k-1})}$ and a sequence $(p_l)_{l\in\N}$ in $\D_k^{(I_{k})}$ with $p_l\to p$ and applying the above formula to $U$ at $p_l\in\D_k^{(I_{k})}$, the big bracket is controlled by $p_l^{i_k}\to 0$ while approaching~$p$ and -- with the derivatives of $U$ inside being bounded on a closed neighbourhood of~$p$ because of the regularity of $U$ -- is continuous up to~$p$. Likewise, all derivatives of $U$ within $\D_{k-1}^{(I_{k-1})}$ are continuously matched by the corresponding ones in $\D_k^{(I_{k})}$, thus $L^*_{k-1}(U(p_l))$ is also continuous up to the boundary in~$p$ (as the corresponding coefficients are, too). Hence, the whole expression is continuous up to the boundary in~$p$ with $L^*_{k-1} U(p)\equiv L^*_{k} U(p)=0$, and since~$p$ was arbitrary, this applies to all of $\D_{k-1}^{(I_{k-1})}$.
\end{proof}


Assuming the stated pathwise regularity, this confines the boundary values of $U$ \resp $f$ on $\bd\D_n=\bigcup_{k=0}^{n-1}\bd_k\D_n$ and correspondingly, equation~\eqref{eq_back_n_stat} is rather restated as an \textit{extended homogeneous} or \textit{extended stationary \KBE\footnote{As already stated, it is without effect whether $\bd_0\D_n$ is added to the domain of definition of the differential equation or not. Although $\bd_0\D_n$ has been included in equation~\eqref{eq_L_closure}, this is not done here for formal reasons.}}

\begin{equation}\label{eq_back_n_stat_ext}
\begin{cases}
L^* U(p)=0 		&\text{in $\overline{\Delta}_n\smin\bd_0\D_n$}\\
U(p) = f(p)	&\text{in $\bd_0\D_n$}\\
\end{cases}
\end{equation}
for $U\in C_p^2\big(\cl{\D}_n\big)$ with the only `free' boundary values remaining the ones on the vertices $\bd_0\D_n$.
If we also assume global continuity of the solution, the values on $\bd_0\D_n$, however, suffice as boundary information determining a solution uniquely. In such a case, a stationary solution and the stationary component of a global extension as in the preceding section also coincide:

\begin{prop}\label{prop_simp_sol}
A solution  
$U\in C_p^\infty\big(\cl{\D}_n\big)\cap C^0\big(\cl{\D}_n\big)$ of the extended stationary \KBE~\eqref{eq_back_n_stat_ext} for some boundary condition $f_0\co \bd_0\D_n\map\R$ 
is uniquely defined and coincides with (the projection of) a solution of the extended \KBE~\eqref{eq_back_n_ext} in $\big(\cl{\D}_n\big)_{-\infty}$ to $\cl{\D}_n$ for a final condition 
$f\in\L^2\big(\bigcup_{d=0}^n\bd_d\D_n\big)$ with $f\equiv f_0\ind_{{\bd_0\D_n}}$ as by theorem~\ref{thm_sol_back_n_ext}. 
Furthermore, the space of solutions is spanned by $p^1,\dotsc,p^n$ and~1.
\end{prop}

\begin{rmk}\blue{
The first assertion of the preceding statement resembles the more general proposition~4.2.1 in \cite{epstein2}: Using their terminology, the corners $\bd_0\D_n$ correspond to the \textit{terminal boundary} $\bd{\D_n}_{ter}$, while by construction, the entire boundary is cleanly met by $L^*$; however, the considered function spaces are not completely identical.}
\end{rmk}

\begin{proof}
The first assertion may be shown by a successive application of the maximum principle: In every instance of the domain ${\D_k^{(I_k)}}\subset\bd_k\D_n$ for all $1\leq k \leq n$, the operator $L^*$ is locally uniformly elliptic, and hence, $U\vert_{{\D_k^{(I_k)}}}$ is uniquely defined by its values on $\bd{\D_k^{(I_k)}}$ by virtue of the maximum principle. Applying this consideration successively for $\bd_{0}\D_n,\dotsc,\bd_{n}\D_n=\D_n$ yields the desired global uniqueness.

Next, we will show that a final condition $f=\ind_{\D_0^{(\set{i_0})}}$ for some $i_0\in I_n$ gives rise to an extended solution $\bar{U}(p,t)=\bar{U}(p)=p^{i_0}$ in $\big(\cl{\D}_n\big)_{-\infty}$ \resp $\cl{\D}_n$ proving the second assertion. With $f$ as described, the extended solution (\cf theorem~\ref{thm_sol_back_n_ext}) is solely given by $\bar{U}\equiv\bar{U}_{i_0}$, \ie
\begin{multline}
\bar{U}_{\set{i_0}}(p,t)=
u_{\set{i_0}}(p,t)\ind_{\D_0^{(\set{i_0})}}(p)\\
+\sum_{ 1\leq  d\leq n}\sum_{ i_{1}\in I_n\smin \set{i_0}}\dotsm
\sumstack{i_d\in I_n\smin \set{i_0,\dotsc,i_{d-1}}}U_{\set{i_0}}^{i_0,\dotsc,i_{d}}(p,t)
\ind_{\D_d^{(\set{i_0, \dotsc,i_{d}})}}(p)
\end{multline}
(\cf equation~\eqref{eq_glob_ext}).
Considering an arbitrary $\D_d^{(I_d)}\subset \cl{\D}_n$, $I_d\subset I_n$, we obtain for the restriction of $\bar{U}_{i_0}$ to $\D_d^{(I_d)}$ using equation~\eqref{eq_path_ext} 
\begin{align}\notag
\bar{U}_{\set{i_0}}(p,t)\vert_{\D_d^{(I_d)}}
&=\sum_{ i_{1}\in I_d\smin \set{i_0}}\dotsm\hspace*{-1em}\sumstack{i_d\in\\ I_d\smin\set{i_{0},\dotsc,i_{d-1}}}U_{\set{i_0}}^{i_0,\dotsc,i_{d}}(p,t)\\
&=\sum_{ i_{1}\in I_d\smin \set{i_0}}\dotsm\hspace*{-1em}\sumstack{i_d\in\\ I_d\smin\set{i_{0},\dotsc,i_{d-1}}}
u_{\set{i_0}}(\pi^{i_0,\dotsc,i_d}(p),t)\prod_{j=0}^{d-1} \frac{p^{i_j}}{\sum_{l=j}^d p^{i_l}}
\end{align}
with $u_{\set{i_0}}(\pi^{i_0,\dotsc,i_d}(p),t)\equiv 1$ as $\pi^{i_0,\dotsc,i_d}(p)\in\D_0^{(\set{i_0})}$ for all $p\in\D_d^{(I_d)}$ and $u_{\set{i_0}}=f=1$ in $\big(\D_0^{(\set{i_0})}\big)_{-\infty}$ by assumption. 
Since we have $\sum_{l=0}^d p^{i_l}=1$ in $\D_d^{(I_d)}$, we may replace the expression 
 $\sum_{l=j}^d p^{i_l}$ by $1- \sum_{l=0}^{j-1} p^{i_l}$ and rearrange the sum (by also suppressing the last sum as the index $i_d$ does no longer occur), which yields altogether
\begin{align}\notag
&\bar{U}_{\set{i_0}}(p,t)\vert_{\D_d^{(I_d)}}=\\\label{eq_the_great_big_sum}
& p^{i_0}\Bigg(\sumstack{i_1\in\\ I_d\smin\set{i_0}} \frac{p^{i_1}}{1- p^{i_0}}\dotsm\hspace*{-1pt}%
\Bigg(
\sumstack{i_j\in\\ I_d\smin\set{i_0,\dotsc,i_{j-1}}}\frac{p^{i_j}}{1- \sum_{l=0}^{j-1} p^{i_l}}%
\dotsm\hspace*{-1pt}\Bigg(
\sumstack{i_{d-1}\in\\ I_d\smin\set{i_0,\dotsc,i_{d-2}}}\frac{p^{i_{d-1}}}{1- \sum_{l=0}^{d-2} p^{i_l}}
\Bigg)\hspace*{-3pt}\Bigg)\hspace*{-3pt}\Bigg).
\end{align}
As we have $\frac{p^{i_j}+\dotso+p^{i_{d}}}{1-\sum_{l=0}^{j-1} p^{i_l}}=1$ for $j={d-1},\dotsc,1$, the whole expression reduces to $\bar{U}_{\set{i_0}}(p,t)\vert_{\D_d^{(I_d)}}=p^{i_0}$. Since $\D_d^{(I_d)}$ was arbitrary, we obtain $\bar{U}_{\set{i_0}}(p,t)\equiv\bar{U}_{\set{i_0}}(p) =p^{i_0}$ in the entire $\cl{\D}_n$. 
\end{proof}

Regarding the \txind{probabilistic interpretation}, the extended setting~\eqref{eq_back_n_stat_ext} also matches the considerations of section~\ref{sec_hier_sol_bkw} as equation~\eqref{eq_back_n_stat_ext} may be viewed as the limit equation for $t\to-\infty$ of the extended  \KBE~\eqref{eq_back_n_ext} (which may be shown as previously). 
This is also reflected in proposition~\ref{prop_simp_sol}: For $t\to -\infty$ and any solution, the only \txind{target sets} with persisting attraction are of course the vertices (respectively corresponding to configurations of the model where all but one allele are extinct), and hence the stationary solutions match the stationary components of the global extensions as in theorem~\ref{thm_sol_back_n_ext}, which in turn result from a non-vanish\-ing final condition in $\bd_0\D_n$. Then, every $\D_0^{(\set{i})}\subset\bd_0\D_n$ may give rise to a solution (component) $p^i$ -- in particular yielding a positive target hit probability on the entire $\D_n$ for all times. However, it is still noted that even the stationary component of solutions as in theorem~\ref{thm_sol_back_n_ext} may in principle be perceived as time-dependent and also describing the transitional attraction of target sets in the entire $\cl{\D}_n$ induced by a given ultimate target set in $\bd_0\D_n$.

In total, proposition~\ref{prop_simp_sol} under the given restrictions thus already yields a full description of the stationary model in the entire $\cl{\D}_n$. However, dropping the global continuity assumption, a much wider class of (stationary) solutions may be observed as described in the preceding section.

\end{document}